\documentclass[a4paper,11pt]{article}

\usepackage{enumerate}
\usepackage{caption}
\usepackage{amsmath}
\usepackage{amsthm}
\usepackage{mathrsfs}
\usepackage{epsfig}
\usepackage{verbatim}
\usepackage[T1]{fontenc}
\usepackage{amsfonts}
\usepackage{float}
\usepackage{caption}
\usepackage{graphicx}
\usepackage{color}
\usepackage{fullpage}

\usepackage[latin2]{inputenc}
\usepackage{url}

\newtheorem{theorem}{Theorem}[section]

\newtheorem{proposition}[theorem]{Proposition}
\newtheorem{lemma}[theorem]{Lemma}
\newtheorem{corollary}[theorem]{Corollary}
\newtheorem{definition}[theorem]{Definition}

\newtheorem{example}[theorem]{Example}
\newtheorem{remark}[theorem]{Remark}

\makeindex

\begin{document}
\title{On measures of accretion and dissipation for solutions of the Camassa-Holm equation\\
\small{}}
\author{{Grzegorz Jamr\'oz} \\
 {\it \small  Institute of Mathematics, Polish Academy of Sciences, ul. \'Sniadeckich 8, 00-656 Warszawa, Poland}\\
{\it \small e-mail: jamroz@impan.pl}
}

\maketitle

\abstract{We investigate the measures of dissipation and accretion related to the weak solutions of the Camassa-Holm equation. Demonstrating certain properties of nonunique characteristics, we prove a new representation formula for these measures and conclude about their structural features, such us singularity with respect to the Lebesgue measure. We apply these results to gain new insights into the structure of weak solutions, proving in particular that measures of accretion vanish for dissipative solutions of the Camassa-Holm equation.}
{\bf MSC Classfication 2010:} 35L65, 37K10

\section{Introduction}
In this paper we study the Camassa-Holm equation (\cite{CH}), 
\begin{equation}
\label{Eq_CHoriginall}
u_t - u_{xxt} + 3uu_x = 2u_x u_{xx} + uu_{xxx},
\end{equation}
which is a one-dimensional model of unidirectional water wave propagation in shallow canals, where $t$ denotes time, $x$ is a one-dimensional space variable and $u(t,x)$ corresponds to the horizontal velocity of the water surface (see \cite{CL}). Equation \eqref{Eq_CHoriginall} had been previously derived, though without the physical context of \cite{CH}, by Fokas and Fuchssteiner \cite{FF1} as a bihamiltonian generalization of the celebrated Korteweg-de Vries equation (KdV) and is formally integrable \cite{ConMac}. The hallmark of the Camassa-Holm equation, which makes it one of the most relevant models of the shallow water theory, is that it accounts both for peaked solitons (\cite{CH}) and, unlike the KdV equation \cite{KPV}, yet similarly as the Whitham equation \cite{Whitham}, for wave-breaking (see \cite{CE3}). 

 The well-posedness theory for \eqref{Eq_CHoriginall} in the case of smooth solutions is due to A. Constantin and J. Escher \cite{CE1,CE2}. The same authors formulated also fairly general conditions precluding wave-breaking \cite{CE3}. Since then, other criteria have been proposed and  recently Brandolese \cite{Brandolese} has obtained a general criterion, which nicely encompasses most previous non-breaking criteria, see references therein. Nevertheless, for general initial data the wave-breaking, understood as blow-up of the  $L^{\infty}$ norm of the derivative $u_x$, is unavoidable. Importantly, however, the Camassa-Holm equation preserves physical relevance also after wave breaking and thus it is worth to study weak (i.e. non-smooth) solutions after the breaking time. This can be done in the framework provided by (with $*$ denoting convolution)
\begin{eqnarray}
\partial_t u + \partial_x (u^2 / 2) + P_x &=& 0, \label{eq_WeakCH1}\\
P(t,x) &=& \frac 1 2 e^{-|x|}* \left(u^2(t,\cdot) + \frac {u_x^2(t,\cdot)}{2}\right),\\
u(t=0, \cdot) &=& u_0, \label{eq_WeakCH3}
\end{eqnarray}
which for smooth solutions is equivalent to \eqref{Eq_CHoriginall}, with equivalence  provided through the nonlocal operator $(I-\partial_{xx})^{-1}$, which satisfies $(I-\partial_{xx})(\frac 1 2 e^{-|x|}) = \delta(x)$. 

Solutions of \eqref{eq_WeakCH1}-\eqref{eq_WeakCH3} are considered in the space $L^{\infty}([0,\infty),H^1(\mathbb{R}))$, which corresponds to the maximal physically relevant class of solutions with bounded total energy, given by
$$E(t) = \frac 1 2 \int_{\mathbb{R}} (u^2(t,x)+u_x^2(t,x)) dx.$$
Let us note that the total energy remains constant for smooth solutions (see e.g.  introduction of \cite{BC2}), yet fails to be so for weak solutions, for which it can both increase and decrease. This allows the weak solutions to account for many phenomena such as soliton interactions, energy dissipation etc.

\begin{definition}[Weak solutions]
Let $u_0 \in H^1(\mathbb{R})$. We say that a function $u: [0,\infty) \times \mathbb{R} \to \mathbb{R}$ is a \emph{weak solution} of \eqref{eq_WeakCH1}-\eqref{eq_WeakCH3} if
\begin{itemize}
\item $u(t,x) \in C([0,\infty) \times \mathbb{R}) \cap L^{\infty}([0,\infty), H^1(\mathbb{R})),$
\item $u(t=0,x) = u_0(x)$ for $x \in \mathbb{R}$,
\item $u(t,x)$ satisfies \eqref{eq_WeakCH1} in the sense of distributions.
\end{itemize}
\end{definition}
Existence of weak solutions of the Camassa-Holm equation was proven by Xin and Zhang in \cite{XZ} by the method of vanishing viscosity, which resulted in the so-called \emph{dissipative weak solutions}.

\begin{definition}[Dissipative weak solutions]
\label{Def_dissipative}
A weak solution of \eqref{eq_WeakCH1}-\eqref{eq_WeakCH3} is called \emph{dissipative} if
\begin{itemize}
\item $\partial_x u(t,x) \le const \left(1+ \frac 1 t\right)$  (Oleinik-type condition)
\item $\|u(t,\cdot)\|_{H^1(\mathbb{R})} \le  \|u(0,\cdot)\|_{H^1(\mathbb{R})}$ for every $t > 0$ (weak energy condition).
\end{itemize}
\end{definition}
Dissipative weak solutions are only one of the many classes of (extremely nonunique) weak solutions of \eqref{eq_WeakCH1}-\eqref{eq_WeakCH3}. Other classes include the conservative  \cite{BC} solutions, which satsisfy a supplementary conservation law ensuring local conservation of the energy, and intermediate \cite{GHR2} solutions, which interpolate between the conservative and dissipative ones. The uniqueness of conservative solutions was demonstrated by Bressan and Fonte \cite{BF} by use of a distance functional related to the optimal transportation problem, see also  \cite{HR2,GHR, GHR1} and \cite{BCZ} for more recent proofs, based on alternative techniques.
The question of uniqueness of dissipative solutions, on the other hand, remained for many years one of the eminent unresolved open problems of the theory and, until recently, only constructions of global semigroups of solutions \cite{BC2,HR} were available. 
The issue has been finally resolved affirmatively in \cite{GJA}, relying on the framework introduced in \cite{CHGJ}, the uniqueness result from \cite{BC2} and ideas similar in spirit to \cite{BCZ}.

In this paper we focus on studying in detail the structure of weak solutions of the Camassa-Holm equation, having in mind another open question of the theory: 
\emph{Is the dissipative weak solution the unique one, which dissipates the energy at the highest possible rate within the class of all weak solutions with the same initial data?} 

The relevance of this question (the so called \emph{maximal dissipation criterion}) in evolutionary equations has been highlighted  by Dafermos \cite{DafMax,DafMAX} and more specifically in the context of the related Hunter-Saxton equation by Zhang and Zheng \cite{ZhangZheng}. The hypothesis of Zhang-Zheng was finally resolved in \cite{TCGJ}, see also references therein, and  it is indeed the framework from \cite{TCGJ}, relying on the ideas of Dafermos regarding generalized characteristics (\cite{DafHS,DafGC}), and ported to the Camassa-Holm setting in \cite{CHGJ}, which provides inspiration for the research in the present paper. At any rate, it is clear that to approach the maximal dissipation criterion, one has to understand the behavior of general weak solutions and this paper is a step in this direction.

To fix some ideas let us in the remainder of this introductory part consider a generic example, which presents possible nonuniqueness scenarios in the class of weak solutions of \eqref{eq_WeakCH1}-\eqref{eq_WeakCH3} and represents one of the few known explicit solutions of the Camassa-Holm equation (see however \cite{Len, Len2}). 
This example provides also some insight into the definition of dissipative weak solutions; it shows in particular that the Oleinik-type condition in Definition \ref{Def_dissipative}, which is an analog of the Oleinik entropy criterion, known from the theory of hyperbolic conservation laws (see e.g. \cite[Section 3b]{EVANS}), is insufficient for selecting the unique dissipative solution and thus the weak energy condition is indispensable.

\begin{example}[Peakon-antipeakon interaction, creation of the peakon-antipeakon pair]
\label{Ex_peakantipeak}
Function
\begin{equation}
\label{Eq_uupeakon}
u(t,x) = p_1(t) e^{-|x-q_1(t)|} - p_1(t)e^{-|x+q_1(t)|}
\end{equation}
where $q(0)<0, p_1(0)>0, p_1(t)= \frac 1 2 p(t), q_1(t) = \frac 1 2 q(t)$ and 
\begin{eqnarray}
p(t) &=& H_0 \frac {[p(0) + H_0]+ [p(0) - H_0]e^{H_0 t}}{[p(0)+H_0] - [p(0)-H_0]e^{H_0 t}}, \label{Eq_formp}\\
q(t) &=& q(0) - 2 \log \frac {[p(0)+H_0]e^{-H_0t \slash 2} + [p(0) - H_0]e^{H_0t \slash 2}}{2p(0)},\nonumber\\
H_0^2 &=& p(0)^2 (1 - e^{q(0)}) = p(t)^2 (1 - e^{q(t)})  \mbox{ for every } t>0 \nonumber
\end{eqnarray}
is, see \cite{BC}, a weak solution of the Camassa-Holm equation with invariant $H_0$. It is composed of two solitary waves -- peakon $p_1(t) e^{-|x-q_1(t)|}$, centred at $q_1(t)$ and moving to the right, and antipeakon $-p_1(t) e^{-|x+q_1(t)|}$ moving to the left (see Fig. \ref{Fig3}). At time $T = \frac 1 {H_0} \log \frac{p(0)+H_0}{p(0)-H_0}$ the two waves annihilate giving rise to a singularity characterized by  
\begin{figure}[h!]
\center
\includegraphics[width=14cm]{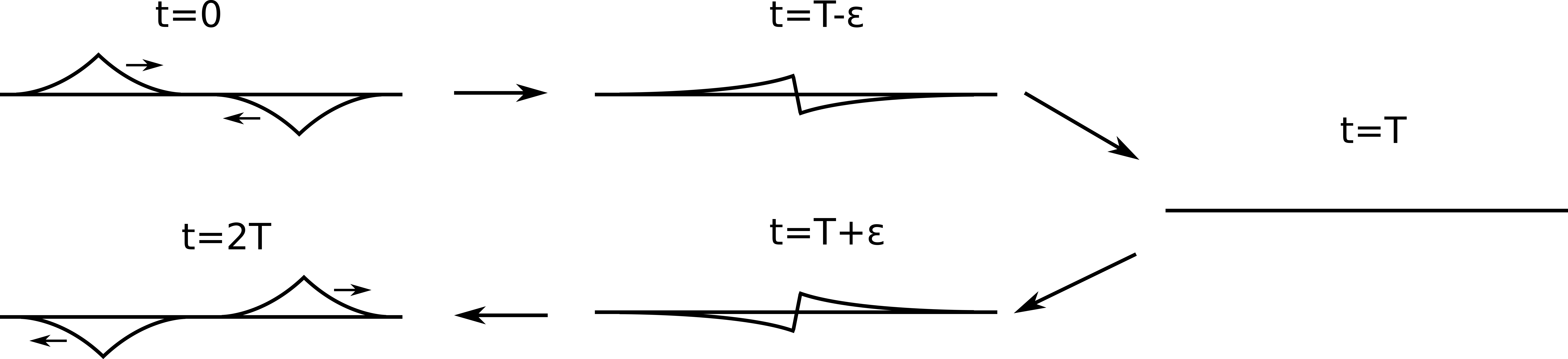}
\caption{Schematic presentation of a conservative  peakon-antipeakon interaction. Two initial waves -- peakon (positive amplitude, moving to the right) and antipeakon (negative amplitude, moving to the left) interact, which leads to their annihilation at the critical time $T$ (wave-breaking). Then they reemerge as a peakon and antipeakon moving away from one another. Considering $T$ as initial time we obtain a weak solution of the Camassa-Holm equation corresponding to creation of a peakon-antipeakon pair out of nothing.}
\label{Fig3}
\end{figure}
\begin{eqnarray*}
\lim_{t \to T^-} \sup_x |u(t,x)| = 0
\end{eqnarray*}
and
\begin{equation}
\label{Eq_e}
\lim_{t \to T^-} e(t,x) = \lim_{t \to T^-} \frac 1 2 (u^2(t,x) + u_x^2(t,x)) = \lim_{t \to T^-} \frac 1 2  ((u_x^-)^2(t,x)) = H_0^2\delta_0(dx),
\end{equation}
where the limits are taken in the weak sense, $u_x^- = \max(0,-u_x)$ is the negative part of $u_x$ and $e(t,\cdot)$ is the energy density. To estimate the derivative of the solution we differentiate \eqref{Eq_uupeakon} obtaining
\begin{equation*}
u_x = -p_1 sgn(x-q_1)e^{-|x-q_1|} + p_1 sgn(x+q_1) e^{-|x+q_1|}
\end{equation*}
and hence $|u_x| \le 2p_1 = p.$
By \eqref{Eq_formp} we have:
\begin{eqnarray*}
\frac {p(t)}{H_0} &=& \frac {[p(0) + H_0]+ [p(0) - H_0]e^{H_0 T}e^{H_0 (t-T)}}{[p(0)+H_0] - [p(0)-H_0]e^{H_0 T}e^{H_0 (t-T)}} = \frac {[p(0) + H_0]+ [p(0) + H_0]e^{H_0 (t-T)}}{[p(0)+H_0] - [p(0)+H_0]e^{H_0 (t-T)}} \\
&=& \frac {1+e^{H_0 (t-T)}}{1 - e^{H_0 (t-T)}} = \frac {e^{H_0 (T-t)} + 1}{e^{H_0 (T-t)}-1} 
 \le \frac {e^{H_0 T}+1}{H_0(T-t)},
\end{eqnarray*}
and consequently 
\begin{equation}
\label{Eq_boundux}
|u_x| \le C(T)/(T-t).
\end{equation}
By setting  
$$u(t,x) = -u(2T- t,x)$$
for $t>T$ we can prolong the solution beyond the blow-up time. This 'conservative' prolongation can be interpreted as reemergence after interaction of the peakon-antipeakon pair, which are moving now away from one another (see Fig. \ref{Fig3}). Note that if we set $u(t,x)=0$ for $t>T, x\in \mathbb{R}$ we would obtain a 'dissipative' prolongation.

Finally, consider $w(t,x):= u(t+T,x)$ defined for $t \ge 0$. As a simple translation in time of $u$, $w$ is a weak solution of \eqref{eq_WeakCH1}-\eqref{eq_WeakCH3}, which satisfies $w(t=0,\cdot) \equiv 0$. Thus, $w$ represents \emph{creation of a peakon-antipeakon pair}.
Due to estimate \eqref{Eq_boundux} we obtain 
\begin{equation*}
|w_x(t,x)| \le \tilde{C}/t,
\end{equation*}
Thus the Oleinik-type criterion from Definition \ref{Def_dissipative} is satisfied. The weak energy condition in Definition \ref{Def_dissipative}, on the other hand, is violated. And indeed, given the initial condition $w(t=0,\cdot)\equiv 0$, the 'entropy solution', called in the context of Camassa-Holm 'dissipative weak solution', is a function equal identically $0$. The weak energy condition from Definition \ref{Def_dissipative} allows us to distinguish between these two scenarios. 
\end{example}
Let us point out that
\begin{itemize}
\item creation of a peakon-antipeakon pair, as presented in Example \ref{Ex_peakantipeak}, can occur at any point of spacetime, which means that weak solutions can, in general, exhibit very complex structures and are highly nonunique,
\item creation of a peakon-antipeakon pair involves creation of a finite portion of energy at a given point in spacetime; similarly, annihilation of a peakon-antipeakon pair corresponds to annihilation of a finite portion of energy,
\item the creation (and annihilation) of a finite portion of energy at any single timepoint can also be spread over the space, leading to a distribution, which, as we will see in the following, can be accounted for by a measure,
\item energy may be also transfered in a continuous fashion (as is the case e.g. for  $t<T$ or $t>T$ in Example \ref{Ex_peakantipeak}). 
\end{itemize}

Our goal in this paper, motivated by the above observations, is to gain new insights into how the energy is accrued/created and dissipated/annihilated in arbitrary  weak solutions of \eqref{eq_WeakCH1}-\eqref{eq_WeakCH3}. 
To this end, we define rigorously and study the so-called \emph{accretion} and \emph{dissipation} measures (introduced informally in \cite{CHGJ}) and obtain some qualitative results regarding their structure. In particular, we show (see Section \ref{Sec_Mresults} for rigorous formulations) that creation of energy at a given timepoint corresponds to a measure which is necessarily singular with respect to the one-dimensional Lebesgue measure and thus cannot be arbitrary. Since, as we demonstrate, the creation of energy for dissipative solutions cannot occur on a singular set, we conclude that for dissipative weak solutions measures of accretion vanish. In the same vein, we show that if a weak solution is to dissipate the energy at the highest possible rate then measures of accretion have to vanish.

The structure of the paper is the following. In Section \ref{Sec_Mresults} we present our main results, recalling first the key results from \cite{CHGJ}. In Section \ref{Sec_Characteristics} we demonstrate some new properties of nonunique characteristics, which we then use in Section \ref{Sec_proofMain} to prove a representation formula for measures of accretion/dissipation.
Finally, in Section \ref{Sec_proofsConclusions} we demonstrate, using the new representation formula, the conclusions regarding the structure 
of weak and dissipative weak solutions of the Camassa-Holm equation.

{\bf Acknowledgements.} This research was partially conducted at the University of Basel in the framework of the Swiss Government Excellence Scholarship for Foreign Scholars and Artists for the Academic Year 2015-2016, grant no. 2015.0079. The author is grateful to Gianluca Crippa from the University of Basel for hospitality. He is furthermore grateful to Tomasz Cie\'slak from the Institute of Mathematics, Polish Academy of Sciences in Warsaw for useful discussions regarding this paper.

\section{Main results}
\label{Sec_Mresults}
In this section we present our main results. Since they rely heavily on the framework developed in \cite{CHGJ}, let us first
recall the key results of \cite{CHGJ}, which, although more general, for the purposes of the present paper have been restricted to the Camassa-Holm equation.

The first proposition (Proposition 2.4 from \cite{CHGJ}, which follows from the Peano existence theorem and \cite[Lemma 3.1]{DafHS}) asserts that characteristics, although in general nonunique, exist.

\begin{proposition}
Let $u$ be a weak solution of \eqref{eq_WeakCH1}-\eqref{eq_WeakCH3}. Then for every $\zeta \in \mathbb{R}$ there exists a (nonunique) characteristic of $u$ emanating from $\zeta$, i.e. a function $\zeta : [0,\infty) \to \mathbb{R}$, which satisfies:
\begin{itemize}
\item $\zeta(0)=\zeta$,
\item $\frac d {dt} \zeta(t) = u(t,\zeta(t)),$
\item $\frac d {dt} u(t,\zeta(t)) = -P_x(t,\zeta(t)).$
\end{itemize}
\end{proposition}

The second result combines Proposition 2.5, Corollary 6.1 and Lemmas 4.1 and 4.2 from \cite{CHGJ} and states that also $u_x$, when restricted to certain set of initial points, evolves along characteristics.

\begin{proposition}[Proposition 2.5, Corollary 6.1 from \cite{CHGJ}]
\label{Prop_MainCH}
Let $u$ be a weak solution of \eqref{eq_WeakCH1}-\eqref{eq_WeakCH3} and fix $t_0 \ge 0$. 
There exists a family of sets $\{S_{t_0,T}\}_{T>t_0}$ such that for every $\zeta \in S_{t_0,T}$ the characteristic $\zeta(\cdot)$, with $\zeta(t_0)=\zeta$, is unique on $[t_0,T]$, $\zeta$ is a Lebesgue point of $u_x(t_0,\cdot)$, and for $t_0 \le t < T$
\begin{equation}
\label{Eq_PropL}
\dot{v}(t) = u^2(t) - \frac 1 2 v^2(t) - P(t),
\end{equation}
where $u(t):=u(t,\zeta(t))$, $v(t) := u_x(t,\zeta(t))$ and $P(t):=P(t,\zeta(t))$.
Moreover, $S_{t_0,T_1} \subset S_{t_0,T_2}$ for $T_1 > T_2$ and $|\mathbb{R} \backslash \bigcup_{T>t_0} S_{t_0,T}| = 0$. Finally, for every $\zeta \in S_{t_0,T}$ there exists $N>0$ such that $|v(\cdot)| \le N$ on $[t_0,T]$. 
\end{proposition}
\begin{remark}
Equation \eqref{Eq_PropL} is satisfied in the sense 
\begin{equation*}
v(t)-v(t_0) = \int_{t_0}^t \left( u^2(s)-\frac 1 2 v^2(s) - P(s)\right) ds
\end{equation*}
for almost every $t \in (t_0,T]$, such that $\zeta(t)$ is a Lebesgue point of $u_x(t,\cdot)$ (in other words, difference quotients converge to $u_x(t,\zeta(t))$). Since the set of Lebesgue points is for every fixed $t$ a full measure set, there exists a modification of $u_x(t,x)$ (for every $t$ on a set of $x$ of measure $0$) such that equation \eqref{Eq_PropL} is satisfied for \emph{every} $t \in [t_0,T]$.
\end{remark}

These two results allowed us in \cite{CHGJ} to establish the main theorem on the weak continuity properties of the positive and negative parts of the derviative $u_x$. Below, by $u_x^{\pm}:= \max(\pm u_x,0)$ we denote the positive/negative parts of function $u_x$.

\begin{theorem}[see Theorem 2.9 in \cite{CHGJ} and Remark \ref{Rem_missing}]
\label{Th_Cadlag}
Let $u$ be a weak solution of the Camassa-Holm equation. Then
\begin{itemize}
\item function $t \mapsto (u_x^+(t,\cdot))^2$ is weakly ladcag (left-continuous with right limits), 
\item function $t \mapsto (u_x^-(t,\cdot))^2$ is weakly cadlag (right-continuous with left limits),
\item the functions $t \mapsto \int_{\mathbb{R}} \phi(x)(u_x^+)^2(t,x)dx$ and $t \mapsto \int_{\mathbb{R}} \phi(x)(u_x^-)^2(t,x)dx$ have locally bounded variation provided $\phi$ is Lipschitz continous,
\item
the limits $$\phi \mapsto \lim_{t \to t_0^+} \int_{\mathbb{R}} \phi(x)(u_x^+)^2(t,x) dx$$ and $$\phi \mapsto  \lim_{t \to t_0^-} \int_{\mathbb{R}} \phi(x) (u_x^-)^2(t,x) dx$$ define, for every fixed $t_0\ge 0$, bounded linear functionals on $C_c(\mathbb{R})$. 
\end{itemize} 

\end{theorem}

As mentioned in \cite{CHGJ}, applying Theorem \ref{Th_Cadlag} we can, using the Riesz representation theorem, define the objects of interest of the present paper -- measures of accretion and dissipation. Their rigorous definition is the point of departure of the present paper.

\begin{definition}[Measures of accretion and dissipation]
\label{Def_MAD}
Let $u$ be a weak solution of \eqref{eq_WeakCH1}-\eqref{eq_WeakCH3}. The \emph{measure of accretion} $\mu^+$ is defined for every $t_0 \in [0,\infty)$ as
\begin{equation*}
\int_{\mathbb{R}} \phi(x) \mu^+(t_0,dx) :=  \lim_{t \to t_0^+} \int_{\mathbb{R}} \phi(x) (u_x^+)^2(t,x) dx  -  \int_{\mathbb{R}} \phi(x) (u_x^+)^2(t_0,x) dx,
\end{equation*}
for any continuous compactly supported function $\phi$. Similarly, for every $t_0 \in (0,\infty)$ the \emph{measure of dissipation} $\mu^-$ is defined as
\begin{equation*}
\int_{\mathbb{R}} \phi(x) \mu^-(t_0,dx) :=  \int_{\mathbb{R}} \phi(x) (u_x^-)^2(t_0,x)dx -  \lim_{t \to t_0^-} \int_{\mathbb{R}} \phi(x) (u_x^-)^2(t,x) dx ,
\end{equation*}
for any continuous compactly supported function $\phi$.
\end{definition}
\begin{remark}
Measures obtained by the use of the Riesz theorem are in general signed. Due to, however, the last inequality from the proof of \cite[Theorem 2.9]{CHGJ} (see also Remark \ref{Rem_missing} below) 
\begin{equation*}
\lim_{t \to t_0^+} \int_{\mathbb{R}} \phi(x) (u_x^+)^2(t,x) dx  -  \int_{\mathbb{R}} \phi(x) (u_x^+)^2(t_0,x) dx
\end{equation*}
is always nonnegative. Thus, $\mu^+(t_0,dx)$ is in fact a nonnegative measure. Similarly, $\mu^-(t_0,dx)$ represents, for any $t_0$, a nonpositive measure. 
\end{remark}
\begin{remark}
\label{Rem_missing}
Note the missing factor $\sup(\phi)$ in the last inequality in \cite{CHGJ}. The correct form of this inequality is
$\int_{\mathbb{R}} \phi u_x^+ (t_2,z)^2 dz  \ge  \int_{\mathbb{R}} \phi u_x^+(t_1,\gamma)^2 d\gamma -(t_2 - t_1){{\sup(\phi)}} (K(a_I(t_1) - a_0(t_1)) + \int_{[a_0(t_1),a_I(t_1)]} u_x(t_1,\gamma)^2 d\gamma ) - ({Lip(\phi) \sup(u) (t_2 - t_1)}  ) \int_{\mathbb{R}} u_x^2(t_2,z)dz$.
Moreover, this form is valid for $\phi$ which are  Lipschitz continuous. For general $\phi \in C_c(\mathbb{R})$ one has to replace the term ${Lip(\phi) \sup(u) (t_2 - t_1)}$ by ${{MC^{\phi}(\sup(u) (t_2 - t_1))}}$, where $MC^{\phi}$ is the modulus of continuity of $\phi$. This means that one obtains, by \cite[Lemma 9.1]{CHGJ}, $BV$ regularity of $\phi \mapsto \int_{\mathbb{R}} \phi(x)(u_x^+)^2(t,x) dx$ only for Lipschitz continuous $\phi$, as formulated here in Theorem \ref{Th_Cadlag}. Nevertheless, the right limits of $\phi \mapsto  \int_{\mathbb{R}} \phi(x)(u_x^+)^2(t,x) dx$ exist, by \cite[Lemma 9.1]{CHGJ}, for arbitrary $\phi \in C_c(\mathbb{R})$.
\end{remark}

\begin{remark}
Measures of accretion and dissipation are defined here for every fixed $t_0$ separately and thus cannot account for accretion/dissipation of energy spread in time. Nevertheless, it seems to be possible to define such measures (at least in the case of dissipative solutions) as two-dimensional objects dependent on both time and space \cite{GJinprogress}.  
\end{remark}

\begin{example}
Due to equality \eqref{Eq_e}, the measures of accretion and dissipation for the peakon-antipeakon interaction from Example \ref{Ex_peakantipeak} are given by 
\begin{eqnarray*}
\mu^-(T,dx) &=& -2H_0^2 \delta_0(dx),\\
\mu^+(T,dx) &=& 2H_0^2 \delta_0(dx), \mbox{ (conservative prolongation)}\\
\mu^+(T,dx) &=& 0. \qquad\qquad\mbox{ (dissipative prolongation)} 
\end{eqnarray*} 
For $t_0 \neq T$, on the other hand, we have, due to smooth evolution in the neighbourhood of $t_0$, that $\mu^+(t_0,dx) =  \mu^-(t_0,dx) = 0$. Note also that for conservative solutions $\mu^+(T,dx)+\mu^-(T,dx) = 0$ and measure $\mu^+(T,dx)$ corresponds to the singular part of the measure used to define a global semigroup of solutions in \cite{BC}.
\end{example}

The main technical result of this paper, based on meticulous studies of characteristics, states that another characterization of measures $\mu^{\pm}$ is possible. Let us begin with a definition.
\begin{definition}[Thick pushforward] 
\label{Def_Thick}
Let $u$ be a weak solution of \eqref{eq_WeakCH1}-\eqref{eq_WeakCH3}. For any Borel set $B \subset \mathbb{R}$ the \emph{thick pushforward (pushbackward)} of $B$ from $t_0$ to $t$ with $t>t_0$ ($t<t_0$) is defined as
\begin{equation*}
B(t):= \{\alpha(t): \alpha(\cdot) \mbox{ is any characteristic of u satisfying } \alpha(t_0) \in B \}.
\end{equation*}
\end{definition}

\begin{theorem}[Characterization of accretion measure]
\label{thB}
Let $u$ be a weak solution of \eqref{eq_WeakCH1}-\eqref{eq_WeakCH3}. Let $t_0\ge 0$ and let $B$ be an arbitrary bounded Borel subset of $\mathbb{R}$.
Then 
\begin{equation*}
\mu^+(t_0,B) = \lim_{t \to t_0^+} \int_{B(t)} (u_x^+)^2(t,x) dx - \int_B (u_x^+)^2(t_0,x) dx,
\end{equation*} 
where $B(t)$ is the thick pushforward of $B$ from $t_0$ to $t$ (see Definition \ref{Def_Thick}).
\end{theorem}
Let us mention that if $u$ is a weak solution of \eqref{eq_WeakCH1}-\eqref{eq_WeakCH3} then for every $t_0>0$ the function $u^{t_0b}(t,x):=-u(t_0-t,x)$ is also a weak solution of \eqref{eq_WeakCH1}-\eqref{eq_WeakCH3}. 
Applying Theorem \ref{thB} to $u^{t_0b}$, we  obtain a dual characterization of the dissipation measure.
\begin{theorem}[Characterization of dissipation measure]
\label{thBbis}
Let $u$ be a weak solution of \eqref{eq_WeakCH1}-\eqref{eq_WeakCH3}. Let $t_0> 0$ and let $B$ be an arbitrary bounded Borel subset of $\mathbb{R}$.
Then 
\begin{equation*}
\mu^-(t_0,B) = \int_B (u_x^-)^2(t_0,x) dx- \lim_{t \to t_0^-} \int_{B(t)} (u_x^-)^2(t,x) dx,
\end{equation*} 
where $B(t)$ is the thick pushbackward of $B$ from $t_0$ to $t$ (see Definition \ref{Def_Thick}).
\end{theorem}

Some of the conclusions from the theory presented above, which give a deeper insight into the structure of weak solutions of the Camassa-Holm equation, are the following.
\begin{theorem}
\label{Th_muac0}
Let $u$ be a weak solution of \eqref{eq_WeakCH1}-\eqref{eq_WeakCH3}. Then for every $t_0 \ge 0$ we have  
$$(\mu^+(t_0,dx))^{ac} = 0,$$ where $(\mu^+(t_0,dx))^{ac}$ is the absolutely continuous, with respect to the one-dimensional Lebesgue measure, part of measure $\mu^+(t_0,dx)$.
Similarly, for every $t_0>0$
$$(\mu^-(t_0,dx))^{ac} = 0.$$
\end{theorem}
\begin{theorem}
\label{th_mu0dis}
Let $u$ be a dissipative weak solution of \eqref{eq_WeakCH1}-\eqref{eq_WeakCH3}. Then $\mu^+(t,dx) = 0$ for every $t \ge 0$.

\end{theorem}

\begin{theorem}
\label{Th_countably}
Let $u$ be a weak solution of \eqref{eq_WeakCH1}-\eqref{eq_WeakCH3}. Then $\mu^+(t_0,dx) \neq 0$ for at most countably many $t_0 \ge 0$. Similarly, $\mu^-(t_0,dx) \neq 0$ for at most countably many $t_0 > 0$.
\end{theorem}

\begin{theorem}
\label{Th_nomaxdissip}
Let $u$ be a weak solution of \eqref{eq_WeakCH1}-\eqref{eq_WeakCH3}. Suppose $\mu^+(t_0,dx) \neq 0$ for some $t_0 \in [0,\infty)$. Then there exists a weak solution $\bar{u}$ of \eqref{eq_WeakCH1}-\eqref{eq_WeakCH3} such that 
\begin{itemize}
\item $\bar{u}(t,\cdot)=u(t,\cdot)$ for $t \in [0,t_0]$,
\item $\limsup_{t \to t_0^+} E(\bar{u}(t,\cdot)) < \liminf_{t \to t_0^+}E(u(t,\cdot))$,
\end{itemize}
where $E(u(t,\cdot)) = \int_{\mathbb{R}} \frac 1 2 (u^2(t,x)+u_x^2(t,x))dx$. 
Consequently, if $\mu^+(t_0,dx)\neq 0$ for some $t_0 \ge 0$ then $u$ does not dissipate energy at the highest possible rate.
\end{theorem}

\section{Preliminaries on characteristics}
\label{Sec_Characteristics}
In this section we prove some technical results regarding  characteristics of solutions of the Camassa-Holm equation, which are crucial in the proof of Theorem \ref{thB}.

Let us begin by recalling a result which asserts that supremum and infimum of a family of characteristics is a characteristic. The formulation below is a slight generalization, to initial points in a bounded set, of \cite[Lemma 5.1]{CHGJ}. Nevertheless, the proof of Lemma \ref{Lem_supinfChar} below follows exactly the proof of \cite[Lemma 5.1]{CHGJ} and thus we do not repeat it.
\begin{lemma}
\label{Lem_supinfChar}
Let $u:[t_0,T] \times \mathbb{R} \to \mathbb{R}$ be a locally bounded continuous function. Let $\{x_\alpha\}_{\alpha \in A}$ be a family of functions satisfying, for $t \in [t_0,T]$, 
\begin{eqnarray*}
\dot{x}_{\alpha}(t)&=&u(t,x_{\alpha}(t))
\end{eqnarray*}
and such that the set  $\{x_{\alpha}(t_0): \alpha \in A\}$ is bounded. Then function $y(t):=\sup_{\alpha \in A} x_{\alpha}(t)$ satisfies $\dot{y}(t)=u(t,y(t))$ and, similarly, function $z(t):= \inf_{\alpha \in A} x_{\alpha}(t)$ satisfies $\dot{z}(t)=u(t,z(t))$. If $u$ is a weak solution of \eqref{eq_WeakCH1}-\eqref{eq_WeakCH3} then $y(t)$ and $z(t)$ are, by \cite[Lemma 3.1]{DafHS}, characteristics of $u$.
\end{lemma}

Using Lemma \ref{Lem_supinfChar} we can define rightmost and leftmost characteristics. Namely, given a weak solution $u$ of \eqref{eq_WeakCH1}-\eqref{eq_WeakCH3}, for every $t_0 \ge 0$ and $\zeta \in \mathbb{R}$ there exist (see \cite[Corollary 5.2]{CHGJ}) the \emph{rightmost} characteristic $\zeta^r(\cdot)$ and the \emph{leftmost} characteristic $\zeta^l(\cdot)$, which are the unique characteristics defined on $[t_0,\infty)$ satisfying 
\begin{itemize}
\item $\zeta^r(t_0) = \zeta = \zeta^l(t_0),$
\item $\zeta^l(t) \le \zeta(t) \le \zeta^r(t)$ for $t \in [t_0,\infty)$ and every characteristic $\zeta(\cdot)$ with $\zeta(t_0)=\zeta$.
\end{itemize}
Similarly, for every $\Gamma \in \mathbb{R}$ there exist 
\emph{rightmost backward} characteristic $\Gamma^{rb}(\cdot)$ and the \emph{leftmost backward} characteristic $\Gamma^{lb}(\cdot)$, which are the unique characteristics defined on $[0,t_0]$ satisfying 
\begin{itemize}
\item $\Gamma^{rb}(t_0) = \Gamma = \Gamma^{lb}(t_0),$
\item $\Gamma^{lb}(t) \le \Gamma(t) \le \Gamma^{rb}(t)$ for $t \in [0,t_0]$ and every characteristic $\Gamma(\cdot)$ with $\Gamma(t_0)=\Gamma$.
\end{itemize}
We finish the introductory part of this section by presenting a change of variables formula, which is indispensable in the theory developed in this paper. Before, however, let us define a useful family of sets (compare \cite[Definition 5.8]{CHGJ}).
\begin{definition}
\label{Def_Lt}
\begin{eqnarray*}
L_{t_0,T}^{unique,N}&:=& \{\zeta \in \mathbb{R}: \zeta(\cdot) \mbox{ is unique forwards on $[t_0,T]$, } \\ && \zeta(s) \mbox{ is a Lebesgue point of } u_x(s,\cdot) \mbox{ for almost every } s \in [t_0,T] \mbox{ and }\\ &&\forall_{\eta \in (\zeta- \frac 1 N, \zeta) \cup (\zeta,\zeta+ \frac 1 N),  s \in [t_0,T]} -N \le \omega(s) \le N \} ,\\
L_{t_0,T}^{unique} &:=& \bigcup_{N=1}^{\infty}L_{t_0,T}^{unique,N},\\
L_{T}^{unique,N}&:=&L_{0,T}^{unique,N},\\
L_{T}^{unique}&:=&L_{0,T}^{unique},
\end{eqnarray*}
where $\omega(s):=\frac{u(s,\eta(s))-u(s,\zeta(s))}{\eta(s)-\zeta(s)}$, $\eta(\cdot)$ is any characteristic satisfying $\eta(t_0)=\eta$ and \emph{unique forwards} means that if $\zeta_1(\cdot)$ and $\zeta_2(\cdot)$ are two characteristics satisfying $\zeta_1(t_0)=\zeta_2(t_0)=\zeta$ then $\zeta_1(s)=\zeta_2(s)$ for all $s \ge t_0$.
\end{definition}
\begin{remark}
\begin{enumerate}[i)]
\item By Fubini theorem the condition that $\zeta(s)$ is a Lebesgue point of $u_x(s,\cdot)$ for almost every $s \in [t_0,T]$ defines a full-measure subset, which however does not have to be Borel. Nevertheless, after removal of a set of measure $0$ the sets $L_{t_0,T}^{unique,N}$ become Borel measurable. 
\item The sets $S_{t_0,T}$ in Proposition \ref{Prop_MainCH} can be chosen as $S_{t_0,T} = L_{t_0,T}^{unique} \backslash Z_{t_0,T}$ for some sets $Z_{t_0,T}$ of Lebesgue measure $0$.
\end{enumerate}
\end{remark}

The following proposition is a generalization to arbitrary initial times $t_0$ of the theory presented in \cite[Section 7]{CHGJ}, with the change of variables formula being a consequence of \cite[(6)]{FT}.
\begin{proposition}[Change of variables formula]
\label{Prop_ChangeOfV}
Fix $t_0 \ge 0$. Let $g$ be a bounded nonnegative Borel measurable function and let $A \subset L_{t_0,t}^{unique}$ be a Borel set. Then 
\begin{equation*}
\int_{M_{t-t_0}(A)} g(z)dz = \int_A g(M_{t-t_0}(\zeta))M_{t-t_0}'(\zeta)d\zeta,
\end{equation*}
where $t \ge t_0$, $M_{t-t_0}(\zeta):= \zeta^l(t)$ and $\zeta^l(\cdot)$ is the leftmost characteristic satisfying $\zeta^l(t_0)=\zeta$. Moreover, for $\zeta \in L_{t_0,t}^{unique}$ 
\begin{equation*}
M_{t-t_0}'(\zeta) = e^{\int_{t_0}^t v(s)ds},
\end{equation*}
where $v(s)=u_x(s,\zeta(s))$.
\end{proposition}
\begin{remark}
By considering sets $A \cap L_{t_0,t}^{unique,N}$ instead of $A$ and passing to the limit $N\to \infty$, it suffices to assume that $g$ in Proposition \ref{Prop_ChangeOfV} is bounded on $L_{t_0,t}^{unique,N}$ for every $N=1,2,\dots$ with a bound possibly dependent on $N$.
\end{remark}

Now, let us present our new results regarding characteristics. The first one shows that rightmost characteristics converge from the right to a rightmost characteristic.
\begin{lemma}
\label{Lem_CDeltaChar}
Let $u$ be a weak solution of \eqref{eq_WeakCH1}-\eqref{eq_WeakCH3} and fix $t_0 \ge 0$. For every $\beta \in \mathbb{R}$ and $t_1 \in [t_0,\infty)$ we have
\begin{equation*}
\lim_{\delta \to 0^+} (\beta+\delta)^r(t_1) = \beta^r(t_1),
\end{equation*}  
where $\beta^r(\cdot)$ and $(\beta+\delta)^r(\cdot)$ are the rightmost characteristics with $\beta^r(t_0) = \beta$ and $(\beta+\delta)^r(t_0) = \beta + \delta$, respectively. 
\end{lemma}
\begin{proof}
Let $\Gamma:=\lim_{\delta \to 0^+} (\beta+\delta)^r(t_1)$. Clearly, $\Gamma\ge\beta^r(t_1)$ since $(\beta+\delta)^r(t_1) \ge \beta^r(t_1)$ for every $\delta>0$.
Suppose $\Gamma > \beta^r(t_1)$. Let $\Gamma^{lb}(t)$ be the leftmost backward characteristic with $\Gamma^{lb}(t_1) = \Gamma$. 
Suppose $\Gamma^{lb}(t_0)\le \beta$. Then there would exist a characteristic emanating from $\beta$ which is more to the right than $\beta^r$, obtained as a concatenation of $\beta^r(\cdot)$ (until the crossing time) and $\Gamma^{lb}(\cdot)$ (from the crossing time on) (see Figure \ref{Fig4}left).
\begin{figure}[h!]
\center
\includegraphics[width=10cm]{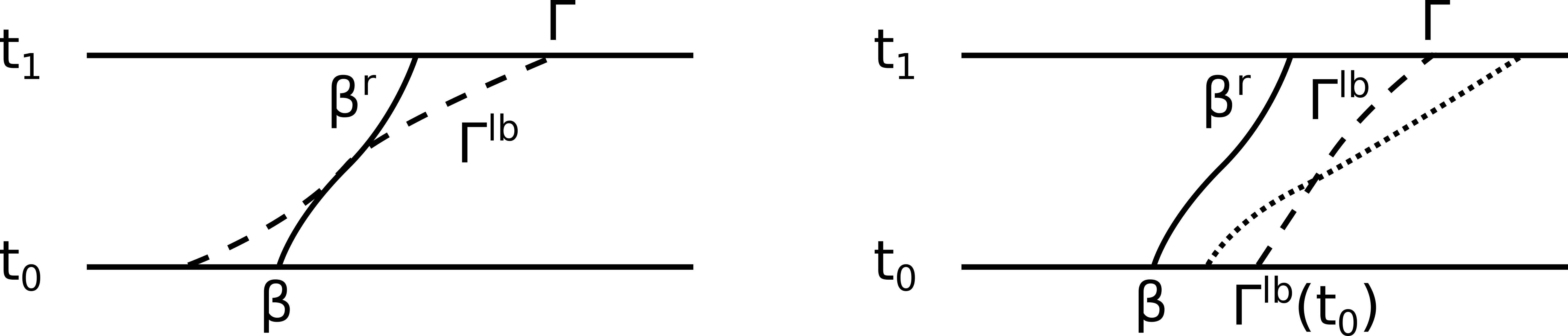}
\caption{Illustration of the two cases leading to a contradiction in the proof of Lemma \ref{Lem_CDeltaChar} for $\Gamma>\beta^r(t_1)$. Left: case $\Gamma^{lb}(t_0) \le \beta$. Then there exists a characteristic emanating from $\beta$ which is more to the right than $\beta^r$. Right: case $\Gamma^{lb}(t_0)>\beta$. Then the rightmost characteristic emanating from a point between $\beta$ and $\Gamma^{lb}(t_0)$ (dotted line) has to finish at a point to the right of $\Gamma$ at time $t_1$ (due to definition of $\Gamma$). Thus it has to cross $\Gamma^{lb}$, which allows us to find a backward characteristic finishing at $\Gamma$, which lies more to the left than $\Gamma^{lb}$.}
\label{Fig4}
\end{figure}

Hence, $\Gamma^{lb}(t_0)>\beta$. Then, however, the characteristic $[(\Gamma^{lb}(t_0)+\beta)/2]^r (\cdot)$ has to cross $\Gamma^{lb}(\cdot)$ on time interval $[t_0,t_1]$ due to definition of $\Gamma$ (see Fig. \ref{Fig4}right). Hence, there exists a characteristic $\Gamma^b(\cdot)$ such that $\Gamma^b(t_1)=\Gamma$ and $\Gamma^b(\cdot)$ is more to the left than $\Gamma^{lb}$ (one needs to take a concatenation of $[(\Gamma^{lb}(t_0)+\beta)/2]^r (\cdot)$ and $\Gamma^{lb}(\cdot)$). This gives contradiction. Hence, $\Gamma=\beta^r(t_1)$.
\end{proof}

The next technical result shows an important convergence property of an integral of $(u_x^+)^2$ in the case when characteristics of a weak solution of the Camassa-Holm equation exhibit behaviour presented in Lemma \ref{Lem_CDeltaChar}.

\begin{lemma}
\label{Lem18}
Let $u$ be a weak solution of \eqref{eq_WeakCH1}-\eqref{eq_WeakCH3} and fix $0 \le t_0 < T$. Let $\beta(\cdot)$ be a characteristic of $u$ and let $\{\beta^\delta\}_{\delta \in \Delta}$ be a collection of characteristics of $u$ such that 
\begin{itemize}
\item $\Delta \subset (0,\infty)$ has an accumulation point $0$, 
\item for every $t_1 \in [t_0,T)$ the function  
$\delta \mapsto \beta^{\delta}(t_1)$ is nondecreasing and 
\begin{equation}
\label{eq_condlim}
\lim_{\delta \to 0^+} \beta^{\delta}(t_1) = \beta(t_1).
\end{equation}
\end{itemize}
Then
\begin{equation}
\label{Eq_cont0}
\lim_{\delta \to 0^+} \left( \limsup_{t \to t_0^+} \int_{(\beta(t), \beta^\delta(t) ]} (u_x^+)^2(t,x) dx\right) = 0.
\end{equation}
Similarly, if $\alpha(\cdot)$ is a characteristic of $u$ and $\{\alpha^\delta\}_{\delta \in \Delta}$ is a family of characteristics such that 
\begin{itemize}
\item $\Delta \subset (0,\infty)$ has an accumulation point $0$, 
\item for every $t_1 \in [t_0,T)$ the function  
$\delta \mapsto \alpha^{\delta}(t_1)$ is nonincreasing and 
\begin{equation}
\label{eq_condlim1}
\lim_{\delta \to 0^+} \alpha^{\delta}(t_1) = \alpha(t_1)
\end{equation}
\end{itemize}
then
\begin{equation}
\label{Eq_cont01}
\lim_{\delta \to 0^+} \left( \limsup_{t \to t_0^+} \int_{[\alpha_\delta(t), \alpha(t) )} (u_x^+)^2(t,x) dx\right) = 0.
\end{equation}

\end{lemma}

\begin{proof}
We only prove \eqref{Eq_cont0}, the proof of \eqref{Eq_cont01} being analogous.

Take arbitrary $\epsilon > 0$ and ${t}_1 \in  (t_0,T)$ such that
\begin{itemize}
\item $t_1 \in \left(t_0,t_0+ \frac \epsilon {100(1+2K)}\right)$ where $K:=\sup(|u|)$ is a bound on the propagation speed of characteristics,
\item  
\begin{equation}
\label{Eq_Omega1}
\Omega(t_1-t_0)\le 1,
\end{equation}
where function $\Omega(\cdot):=\sqrt{LC} \tan(\cdot \sqrt{LC} -\frac \pi 2)$ is discussed in \cite[Definition 5.5]{CHGJ},
\item 
\begin{equation}
\label{eq_sqrt}
\sqrt{LC} \tan (-\sqrt{LC}(t_1-t_0) + \arctan(1/\sqrt{LC}))>1/2,
\end{equation} 
where $L,C$ are certain constants dependent only on the $L^{\infty}([0,\infty), H^1(\mathbb{R}))$ norm of the solution $u$ (see \cite {CHGJ}),
\item $t_1 - t_0< \frac {\epsilon} {4 (\sup(u^2)+\sup(P))}.$
\end{itemize}
Let us comment that the third technical condition means that, due to \cite[(26)]{CHGJ}, if $u_x(t,\zeta) > 1$ for some $t \in [t_0,t_1]$ and if $\zeta \in L_{t,t_1}^{unique} \cap S_{t,t_1}$ then $u_x(s,\zeta(s))>1/2$ for every $s \in [t,t_1]$ such that $u_x(s,\zeta(s))$ exists. Let us also mention that due $L^{\infty}([0,\infty),H^1(\mathbb{R}))$ regularity of $u$, functions $u$ and $P$ are globally bounded and so the quantities $\sup(u^2)$, $\sup(P)$ are well defined.

Let now $\delta$  be so small that $\beta^{\delta}(t_0)-\beta(t_0) < \epsilon/(100(1+2K))$ and
\begin{equation}
\label{Eq_eps100}
\int_{(\beta(t_1),\beta^\delta(t_1)]} (u_x^+)^2(t_1,x)dx < \epsilon/100.
\end{equation}
We will show that 
$$\int_{(\beta(t),\beta^\delta(t)]} (u_x^+)^2(t,x)dx < \epsilon$$
for every $t \in (t_0,t_1)$, see Fig. \ref{FigA}.

\begin{figure}[h!]
\center
\includegraphics[width=4.5cm]{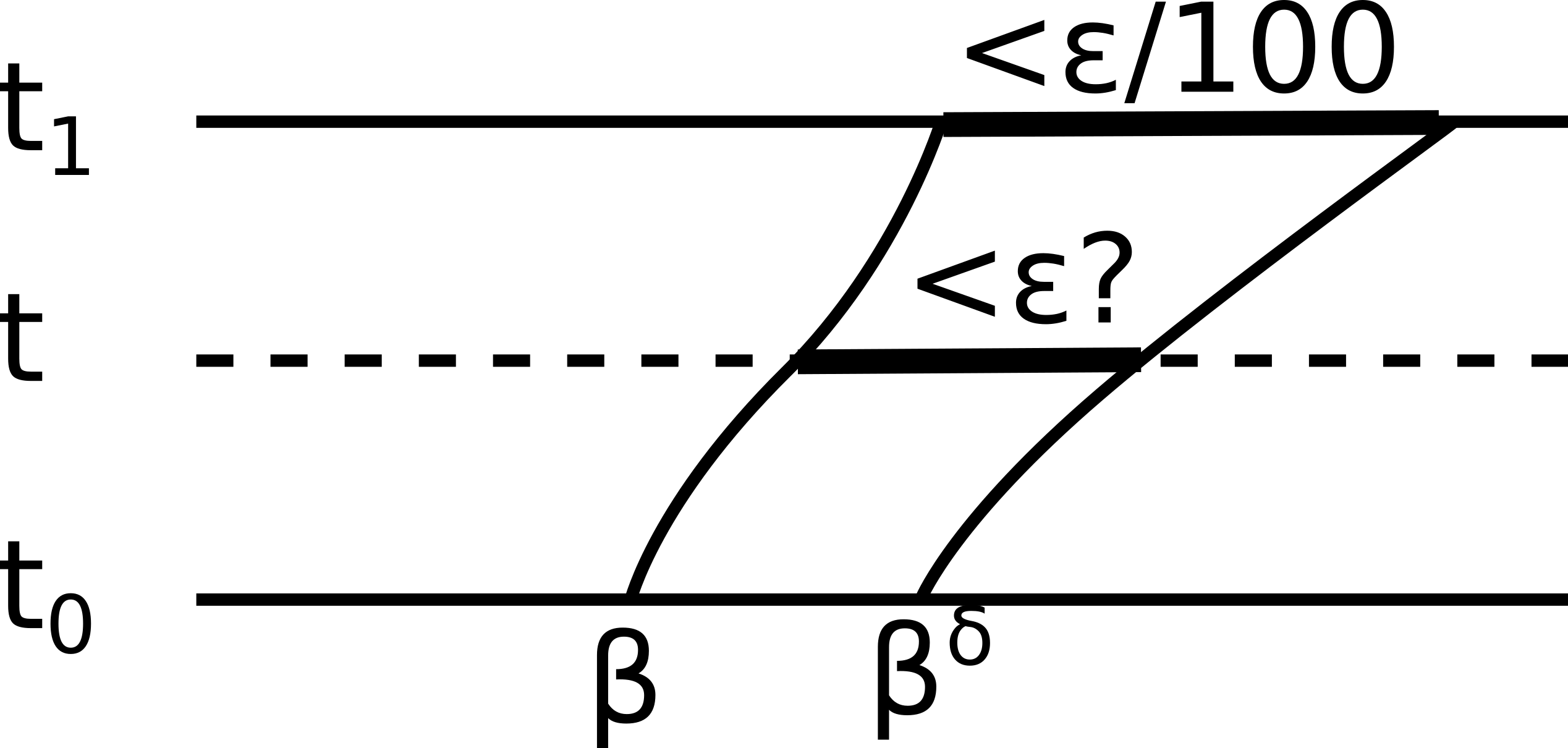}
\caption{Illustration of the proof of Lemma  \ref{Lem18}. We prove that if $t_1$ is sufficiently close to $t_0$ and $\delta$ is small enough then $\int_{(\beta(t_1),\beta^\delta(t_1)]} (u_x^+)^2(t_1,x)dx < \epsilon/100$ implies $\int_{(\beta(t),\beta^\delta(t)]} (u_x^+)^2(t,x)dx < \epsilon$ for every $t \in [t_0,t_1]$.}
\label{FigA}
\end{figure}

Suppose the contrary. Then there exists $t \in (t_0,t_1)$ such that 
$$\int_{(\beta(t),\beta^\delta(t)]} (u_x^+)^2(t,x)dx \ge \epsilon.$$
Due to finite propagation speed of characteristics ($K=\sup(|u|)$) we have 
$$\beta^\delta(t) - \beta(t) < \beta^{\delta}(t_0)-\beta(t_0) + 2K \epsilon/(100(1+2K)) < \epsilon/100.$$
Hence, 
$$\int_{(\beta(t),\beta^\delta(t)] \cap \{
u_x^+(t,\cdot) \ge 1\}} (u_x^+)^2(t,x)dx \ge (99/100)\epsilon.$$

Moreover, $|\{u_x^+(t,\cdot) \ge 1\} \backslash I_{t_1-t}|=0$ due to \eqref{Eq_Omega1}, where $I_{t_1-t}$ is defined in \cite[Definition 5.5]{CHGJ}, and hence by \cite[Lemma 5.7]{CHGJ} and \cite[Lemma 5.9]{CHGJ} we can find $N$ so big that 
$$\int_{(\beta(t),\beta^\delta(t)] \cap \{
u_x^+(t,\cdot) \ge 1\} \cap L_{t,t_1}^{unique,N}\cap S_{t,t_1}} (u_x^+)^2(t,x)dx \ge (98/100)\epsilon.$$

Using now \eqref{eq_sqrt} we obtain that if $\zeta \in (\beta(t),\beta^\delta(t)] \cap \{
u_x^+(t,\cdot) \ge 1\} \cap L_{t,t_1}^{unique,N}\cap S_{t,t_1}$ then $u_x(t_1,\zeta(t_1))>1/2$ and hence, by \cite[Proposition 7.1]{CHGJ}, 
\begin{equation*}
e^{-\epsilon} \le e^{- \frac {2(t_1 - t_0)}{1/2} ((\sup u)^2 + \sup(P))} \le \frac {v^2(t_1)M'_{t_1 - t} (\zeta)}{v^2(t)},
\end{equation*}
where $v(t)=u_x(t,\zeta(t))$ and $M_{t_1-t}(\zeta) = \zeta^l(t_1)$ is the leftmost characteristic satisfying $\zeta^l(t)=\zeta$, which for the range of $\zeta$ that are of interest to us are in fact unique characteristics.

Using now the change of variables formula from Proposition \ref{Prop_ChangeOfV}  
we obtain 
\begin{eqnarray*}
\int_{(\beta(t_1), \beta^\delta(t_1)]} (u_x^+)^2(t_1,z)dz &\ge& \int_{M_{t_1-t} ((\beta(t),\beta^\delta(t)] \cap \{
u_x^+(t,\cdot) \ge 1\} \cap L_{t,t_1}^{unique,N}\cap S_{t,t_1})} (u_x^+)^2(t_1,z) dz \\
&=& \int_{(\beta(t),\beta^\delta(t)] \cap \{
u_x^+(t,\cdot) \ge 1\} \cap L_{t,t_1}^{unique,N}\cap S_{t,t_1}} (u_x^+)^2(t_1,M_{t_1-t}(\zeta)) M'_{t_1 - t}(\zeta) d\zeta \\
&\ge& \int_{(\beta(t),\beta^\delta(t)] \cap \{
u_x^+(t,\cdot) \ge 1\} \cap L_{t,t_1}^{unique,N}\cap S_{t,t_1}} (u_x^+)^2(t,\zeta) e^{-\epsilon} d\zeta \\
&\ge& 98/100 \epsilon e^{-\epsilon} \ge 1/100 \epsilon,
\end{eqnarray*}
which contradicts \eqref{Eq_eps100}.
\end{proof}
\begin{corollary}
\label{Cor_limlim}
Let $u$ be a weak solution of \eqref{eq_WeakCH1}-\eqref{eq_WeakCH3}. Then for every $t_0 \ge 0$ and every $\alpha,\beta \in \mathbb{R}$ we have
\begin{equation}
\label{Eq_cont1}
\lim_{\delta \to 0^+} \left( \limsup_{t \to t_0^+} \int_{(\beta^r(t), (\beta+\delta)^r(t) ]} (u_x^+)^2(t,x) dx\right) = 0,
\end{equation}
\begin{equation}
\label{Eq_cont2}
\lim_{\delta \to 0^+} \left( \limsup_{t \to t_0^+} \int_{[(\alpha-\delta)^l(t), \alpha^l(t))} (u_x^+)^2(t,x) dx\right) = 0.
\end{equation}
\end{corollary}
\begin{proof}
To prove \eqref{Eq_cont1} take $\beta(\cdot) \equiv \beta^r(\cdot)$ and $\beta^{\delta}(\cdot) \equiv (\beta+\delta)^r(\cdot)$ and, using Lemma \ref{Lem_CDeltaChar} apply Lemma \ref{Lem18}. \eqref{Eq_cont2} follows in a similar fashion. 
\end{proof}

Lemma \ref{Lem18} is a direct consequence of an even more general result contained in Lemma \ref{Lem19}, which we now formulate and prove. Let us note that the proof of Lemma \ref{Lem19} follows the lines of the proof of Lemma \ref{Lem18} with a few minor modifications. Since, however, Lemma \ref{Lem19} is much more abstract, we decided to keep both proofs for clarity of the exposition. Importantly, the formulation of Lemma \ref{Lem18} will not always be sufficient for our purposes, and on one occasion we will have to resort to the general formulation given in Lemma \ref{Lem19}.

\begin{lemma}
\label{Lem19}
Let $u$ be a weak solution of \eqref{eq_WeakCH1}-\eqref{eq_WeakCH3} and fix $t_0,T$, which satisfy $0\le t_0 < T$. Let $\{\mathcal{B}^{\delta}\}_{\delta \in \Delta}$ be a family of collections of characteristics on $[t_0,T]$ satisfying 
\begin{itemize}
\item $\Delta \subset (0,\infty)$ has an accumulation point $0$, 
\item $\mathcal{B}^{\delta_1} \subset \mathcal{B}^{\delta_2}$ for $\delta^1 < \delta^2$,
\item for every $t \in [t_0,T]$ the set $B^{\delta}(t):=
\{\beta(t): \beta(\cdot) \in {\mathcal B}^{\delta}\}$ is Borel measurable,
\item for every $t \in [t_0,T]$
\begin{equation*} \lim_{\delta \to 0^+} \int_{B^{\delta}(t)} (u_x^+)^2(t,x) = 0,
\end{equation*}
\item there exists $D>0$ such that for every $t \in [t_0,T]$ we have 
\begin{equation}
\label{Eq_DiamBdelta}
{\rm diam}(B^{\delta}(t))<D.
\end{equation} 
\end{itemize}
Then
\begin{equation*}
\lim_{\delta \to 0^+} \left( \limsup_{t \to t_0^+} \int_{B^{\delta}(t)} (u_x^+)^2(t,x) dx\right) = 0.
\end{equation*}

\end{lemma}
\begin{proof} 
Take arbitrary $\epsilon > 0$ and ${t}_1 \in  (t_0,T)$ such that
\begin{itemize}
\item ${t}_1 \in \left(t_0,t_0+ \frac{\epsilon}{2(\sup(u^2)+\sup(P))} \left(\frac {\epsilon} {200D}\right)^{1/2}\right),$ 
\item  
\begin{equation}
\label{Eq_Omega2}
\Omega(t_1-t_0)\le \left(\frac{\epsilon}{100D}\right)^{1/2},
\end{equation}
with $\Omega$ defined in \cite[Definition 5.5]{CHGJ}
\item $\sqrt{LC} \tan \left(-\sqrt{LC}(t_1-t_0) + \arctan((\epsilon/(100D))^{1/2})/\sqrt{LC})\right)> \left({\epsilon}/{(200D)}\right)^{1/2} $, where $L,C$ are  constants dependent only on the $L^{\infty}([0,\infty), H^1(\mathbb{R}))$ norm of the solution $u$ (see \cite {CHGJ}).
\end{itemize}
Finally, take $\delta$ so small that 
$$\int_{B^{\delta}(t_1)} (u_x^+)^2(t_1,x)dx < \epsilon/100.$$
We will show that 
$$\int_{B^{\delta}(t)} (u_x^+)^2(t,x)dx < \epsilon$$
for every $t \in (t_0,t_1)$.
Suppose the contrary. Then there exists $t \in (t_0,t_1)$ such that 
$$\int_{B^{\delta}(t)} (u_x^+)^2(t,x)dx \ge \epsilon.$$
Since, however, $|B^{\delta}(t)| < D$ we conclude 
$$\int_{B^\delta(t) \cap \{
(u_x^+)^2(t,\cdot) \ge \epsilon/(100D)\}} (u_x^+)^2(t,x)dx \ge (99/100)\epsilon.$$
Moreover, $|\{(u_x^+)^2(t,\cdot) \ge (\epsilon/100D)\} \backslash I_{t_1-t}|=0$ due to \eqref{Eq_Omega2} and hence by \cite[Lemma 5.7]{CHGJ} we can find $N$ so big that 
$$\int_{B^\delta(t) \cap \{
(u_x^+)^2(t,\cdot) \ge \epsilon/(100D)\} \cap L_{t,t_1}^{unique,N} \cap S_{t,t_1}} (u_x^+)^2(t,x)dx \ge (98/100)\epsilon.$$
Using now the fact that if $\zeta \in B^\delta(t) \cap \{
(u_x^+)^2(t,\cdot) \ge \epsilon/(100D)\} \cap L_{t,t_1}^{unique,N}\cap S_{t,t_1}$ then $u_x(t_1,\zeta(t_1))>(\epsilon/(200D))^{1/2}$ (for a.e. $\zeta$) we obtain, by \cite[Proposition 7.1]{CHGJ}, 
\begin{equation*}
e^{-\epsilon} \le e^{- \frac {2(t_1 - t_0)}{(\epsilon/(200D))^{1/2}} ((\sup u)^2 + \sup(P))} \le \frac {v^2(t_1)M'_{t_1 - t} (\zeta)}{v^2(t)}.
\end{equation*}

Using now the change of variables formula from Proposition \ref{Prop_ChangeOfV} we obtain
\begin{eqnarray*}
\int_{B^\delta(t_1)} (u_x^+)^2(t_1,z)dz 
&\ge& \int_{M_{t_1-t} (B^\delta(t) \cap \{
(u_x^+)^2(t,\cdot) \ge \epsilon/(100D)\} \cap L_{t,t_1}^{unique,N}\cap S_{t,t_1})} (u_x^+)^2(t_1,z) dz \\
&=& \int_{B^\delta(t) \cap \{
(u_x^+)^2(t,\cdot) \ge \epsilon/(100D)\} \cap L_{t,t_1}^{unique,N}\cap S_{t,t_1}} (u_x^+)^2(t_1,M_{t_1-t}(\zeta)) M'_{t_1 - t}(\zeta) d\zeta \\
&\ge& \int_{B^\delta(t) \cap \{
(u_x^+)^2(t,\cdot) \ge \epsilon/(100D)\} \cap L_{t,t_1}^{unique,N}\cap S_{t,t_1}} (u_x^+)^2(t,\zeta)) e^{-\epsilon} d\zeta \\
&\ge& (98/100) \epsilon e^{-\epsilon} \ge  \epsilon/100,
\end{eqnarray*}
which gives contradiction.
\end{proof}
\begin{remark}
Lemma \ref{Lem18} follows from Lemma \ref{Lem19} by setting $\mathcal{B}^{\delta} := \{\zeta(\cdot): \beta(t) \le \zeta(t) \le \beta^{\delta}(t) \mbox{ for every } t \in [t_0,T]\}$ and observing that by finite propagation speed of characteristics \eqref{eq_condlim} implies \eqref{Eq_DiamBdelta}.
\end{remark}

\section{Proof of Theorem \ref{thB}}
\label{Sec_proofMain}
To prove Theorem \ref{thB} we show it for the following sequence of classes of $B \subset \mathbb{R}$:
\begin{itemize}
\item $B$ -- closed interval,
\item $B$ -- open interval,
\item $B$ -- arbitrary open set,
\item $B$ -- arbitrary compact set,
\item $B$ -- arbitrary Borel set.
\end{itemize}

\noindent{\bf ${\bf B}$ - closed interval}

Let $B = [\alpha,\beta]$. Then the thick pushforward of $B$ is given by $B(t) = [\alpha^l(t),\beta^r(t)]$, where $\alpha^l(\cdot)$ is the leftmost and $\beta^r(\cdot)$ the rightmost characteristic emanating from $\alpha$ and $\beta$, respectively. To prove Theorem \ref{thB} for $B$ in this form, it suffices thus to show the following result.
\begin{proposition}
\label{Prop_16}
Let $u$ be a weak solution of \eqref{eq_WeakCH1}-\eqref{eq_WeakCH3}. The measure of accretion $\mu^{+}$ associated to $u$ satisfies
\begin{equation}
\label{Eq_altdef}
\mu^+(t_0,[\alpha,\beta]) = \lim_{t \to t_0^+} \int_{[\alpha^l(t),\beta^r(t)]} (u_x^+)^2 (t,x) dx - \int_{[\alpha,\beta]} (u_x^+)^2(t_0,x)dx
\end{equation}
for every $\alpha,\beta \in \mathbb{R}, \alpha \le \beta,$ and $t_0 \in [0,\infty)$. 
\end{proposition}
\begin{proof}
By regularity of the measure $\mu^+$, it suffices to show that
\begin{equation}
\label{Eq_limmuplus}
\lim_{t \to t_0^+}  \int_{[\alpha^l(t),\beta^r(t)]} (u_x^+)^2 (t,x)dx = \lim_{\epsilon \to 0^+} \lim_{t \to t_0^+} \int_{\mathbb{R}} \phi^{\epsilon} (x) (u_x^+)^2(t,x)dx,
\end{equation} 
where (see Fig. \ref{Fig5}left)
\begin{equation*}
\phi^{\epsilon} (x) = 
\begin{cases} 
1 & \mbox{ if } x \in [\alpha - \epsilon, \beta + \epsilon],  \\
0 & \mbox{ if } x \le \alpha - 2 \epsilon \mbox{ or } x \ge \beta + 2 \epsilon, \\
\frac 1 \epsilon (x - (\alpha - 2\epsilon)) & \mbox{ if } x \in (\alpha - 2 \epsilon, \alpha - \epsilon),\\
-\frac 1 \epsilon (x - (\beta + 2\epsilon)) & \mbox{ if } x \in (\beta + \epsilon, \beta + 2\epsilon).
\end{cases}
\end{equation*}
Indeed, by definition of measure $\mu^+$ (Definition \ref{Def_MAD}) we have
\begin{equation*}
\int_{\mathbb{R}} \phi^{\epsilon}(x) d\mu^+(t_0,dx) :=  \lim_{t \to t_0^+} \int_{\mathbb{R}} \phi^{\epsilon}(x) (u_x^+)^2(t,x) dx  -  \int_{\mathbb{R}} \phi^{\epsilon}(x) (u_x^+)^2(t_0,x) dx.
\end{equation*}
Passage to the limit $\epsilon \to 0$ in combination with outer regularity of the measure $\mu^+$ and \eqref{Eq_limmuplus} leads to \eqref{Eq_altdef}.

To prove \eqref{Eq_limmuplus} we observe that by finite propagation speed of characteristics we have $\phi^{\epsilon} \ge \bold{1}_{[\alpha^l(t), \beta^r(t)]}$ for $t$ sufficiently close to $t_0$ (Fig. \ref{Fig5}middle). Hence, for every $\epsilon > 0$, 
\begin{equation*}
\lim_{t \to t_0^+} \int_{[\alpha^l(t),\beta^r(t)]} (u_x^+)^2 (t,x)dx  \le \lim_{t \to t_0^+} \int_{\mathbb{R}} \phi^{\epsilon} (x) (u_x^+)^2 (t,x)dx.
\end{equation*}
and thus
\begin{equation*}
\lim_{t \to t_0^+} \int_{[\alpha^l(t),\beta^r(t)]} (u_x^+)^2 (t,x)dx  \le \lim_{\epsilon \to 0^+}\lim_{t \to t_0^+} \int_{\mathbb{R}} \phi^{\epsilon} (x) (u_x^+)^2 (t,x)dx.
\end{equation*}

\begin{figure}[h!]
\center
\includegraphics[width=15cm]{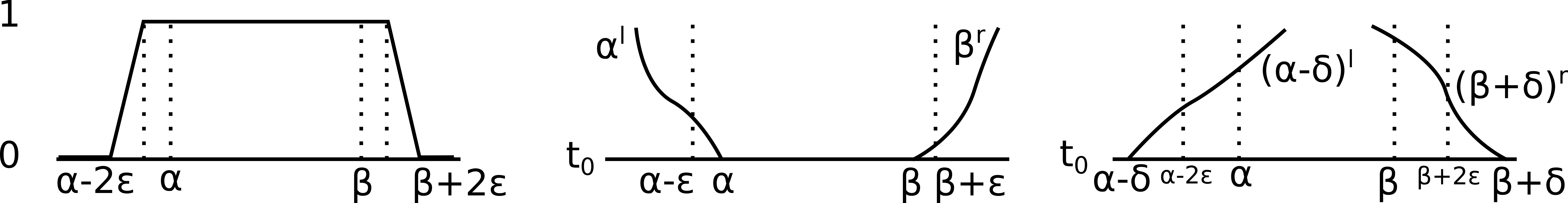}
\caption{Graph of the function $\phi^{\epsilon}$ (left). Illustration of the inequality $\phi^{\epsilon} \ge \bold{1}_{[\alpha^l(t), \beta^r(t)]}$ for $t$ sufficiently close to $t_0$ (middle). Illustration of the fact that for every $\delta>0$ we have $\phi^{\epsilon} \le \bold{1}_{[(\alpha-\delta)^l(t), (\beta+\delta)^r(t)]}$ for $\epsilon$ small enough and $t$ sufficiently close to $t_0$ (right). }
\label{Fig5}
\end{figure}

On the other hand, for every $\delta > 0$ we have (see Fig. \ref{Fig5}right)
\begin{equation*}
\lim_{\epsilon \to 0^+}\lim_{t \to t_0^+} \int_{\mathbb{R}} \phi^{\epsilon} (x) (u_x^+)^2 (t,x)dx \le
\lim_{t \to t_0^+} \int_{[(\alpha-\delta)^l(t),(\beta+\delta)^r(t)]} (u_x^+)^2 (t,x)dx.
\end{equation*}
By Corollary \ref{Cor_limlim} we can pass to the limit $\delta \to 0^+$ to obtain
\begin{equation*}
\lim_{\epsilon \to 0^+}\lim_{t \to t_0^+} \int_{\mathbb{R}} \phi^{\epsilon} (x) (u_x^+)^2 (t,x)dx \le
\lim_{t \to t_0^+} \int_{[\alpha^l(t),\beta^r(t)]} (u_x^+)^2 (t,x)dx.
\end{equation*}
Let us comment that the limits on the right-hand side of the last two inequalities exist by \cite[Theorem 2.7]{CHGJ}.
\end{proof}
\begin{remark}
Propositon \ref{Prop_16} shows that the 'canonical' choice of characteristics in \cite[Theorem 2.8]{CHGJ}, resulting in a measure is the leftmost characteristic emanating from $\alpha$ and the rightmost characteristic emanating from $\beta$.
\end{remark}

\begin{corollary}[of Proposition \ref{Prop_16}]
\label{Cor_muplus}
Let $u$ be a weak solution of \eqref{eq_WeakCH1}-\eqref{eq_WeakCH3}. The measure $\mu^{+}$ satisfies for every $\alpha,\beta \in \mathbb{R}$, $\alpha<\beta$:
\begin{eqnarray*}
\mu^+(t_0,\{\alpha\}) &=& \lim_{t \to t_0^+} \int_{[\alpha^l(t),\alpha^r(t)]} (u_x^+)^2 (t,x) dx,\\
\mu^+(t_0,(\alpha,\beta)) &=& \lim_{t \to t_0^+} \int_{(\alpha^r(t),\beta^l(t))} (u_x^+)^2 (t,x) dx - \int_{(\alpha,\beta)} (u_x^+)^2(t_0,x)dx.
\end{eqnarray*}
\end{corollary}

\noindent{\bf ${\bf B}$ - open interval}

Let us begin by introducing the notions of \emph{left-rightmost characteristic} and \emph{right-leftmost characteristic}.

\begin{definition}
\label{Def_Azeta}
Let $A_{\zeta}$ be a family of characteristics given by
\begin{equation*}
A_{\zeta}:= \{\zeta(\cdot): \zeta(t_0)=\zeta \mbox{ and there exists } t_1>t_0 \mbox{ such that } \zeta(t)=\zeta^r(t) \mbox{ for every } t \in [t_0,t_1] \}.
\end{equation*}
The \emph{left-rightmost characteristic}, $\zeta^{rl}$, is defined, using Lemma \ref{Lem_supinfChar}, by (see Fig. \ref{Fig6})
\begin{equation*}
\zeta^{rl}(t) := \inf \{ \zeta(t): \zeta(\cdot) \in A_{\zeta}\}.
\end{equation*}
Similarly, let $\hat{A}_{\zeta}$ be a family of characteristics given by
\begin{equation*}
\hat{A}_{\zeta}:= \{\zeta(\cdot): \zeta(t_0)=\zeta \mbox{ and there exists } t_1>t_0 \mbox{ such that } \zeta(t)=\zeta^l(t) \mbox{ for every } t \in [t_0,t_1] \}
\end{equation*}
Then \emph{right-leftmost characteristic}, $\zeta^{lr}$, is defined by
\begin{equation*}
\zeta^{lr}(t) := \sup \{ \zeta(t): \zeta(\cdot) \in \hat{A}_{\zeta}\}.
\end{equation*}
\begin{figure}[h!]
\center
\includegraphics[width=4.7cm]{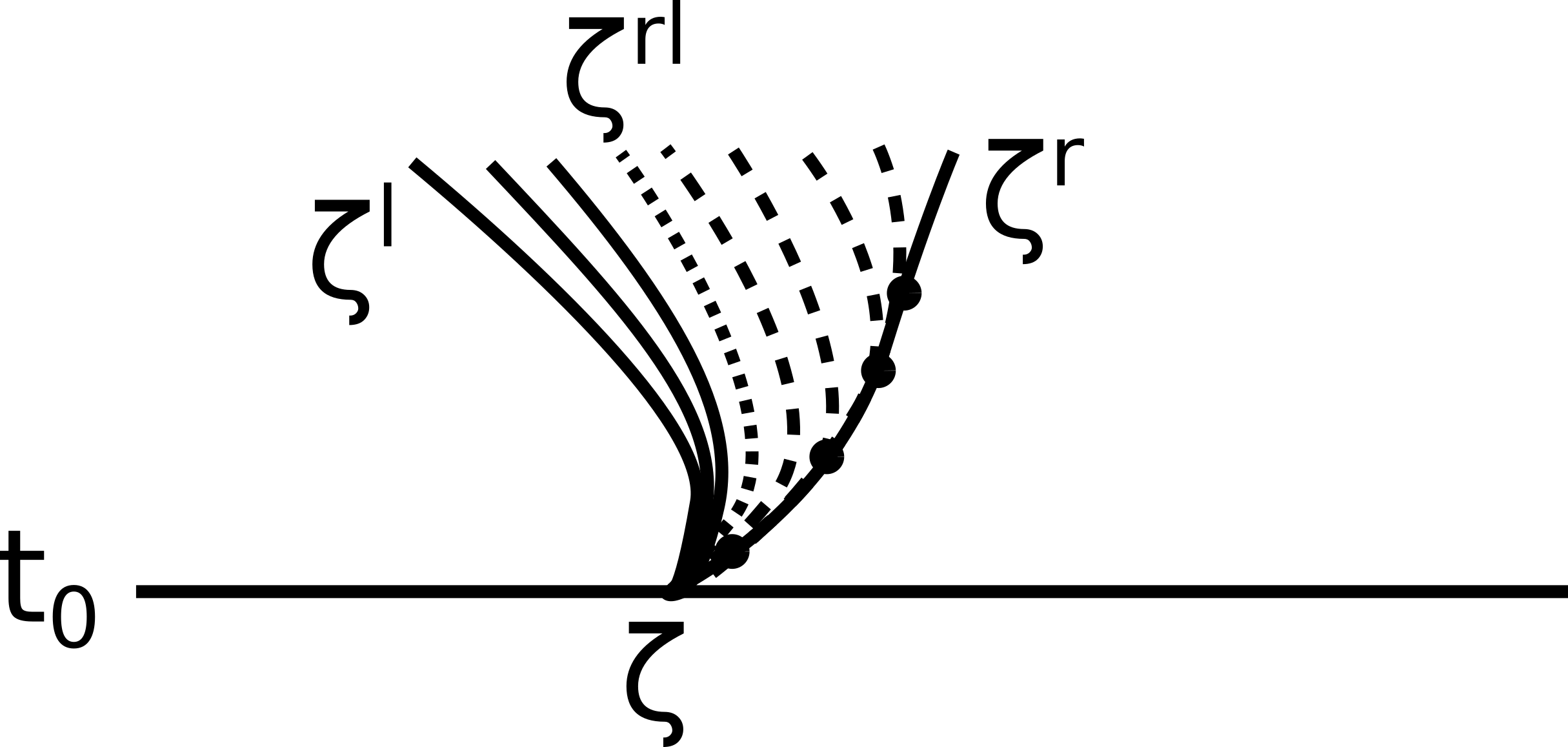}
\caption{Illustration of the concept of the left-rightmost characteristic. The left-rightmost characteristic $\zeta^{rl}(\cdot)$, dotted line, is the infimum of the family (dashed lines) of characteristics which coincide with the rightmost characteristic $\zeta^r(\cdot)$ up to some timepoint $t>t_0$. Since there might exist characteristics emanating from $\zeta$ which are more to the left than $\zeta^{rl}$, the left-rightmost characteristic does not in general coincide with the leftmost characteristic.}
\label{Fig6}
\end{figure}

\end{definition}

\begin{lemma}
\label{Lem_44}
Let $u$ be a weak solution of \eqref{eq_WeakCH1}-\eqref{eq_WeakCH3} and fix $t_0 \ge 0$. Then for every $\zeta \in \mathbb{R}$ we have
\begin{equation}
\label{Eq_contleftright}
\lim_{t \to t_0^+} \int_{[\zeta^{rl}(t), \zeta^r(t) ]} (u_x^+)^2(t,x) dx = 0.
\end{equation}
and, similarly,
\begin{equation}
\label{Eq_contrightleft}
\lim_{t \to t_0^+} \int_{[\zeta^{l}(t), \zeta^{lr}(t) ]} (u_x^+)^2(t,x) dx = 0.
\end{equation}
\end{lemma}
\begin{remark}
The convergences from Lemma \ref{Lem_44} hold in stark contrast to \begin{equation*}
\lim_{t \to t_0^+} \int_{[\zeta^{l}(t), \zeta^r(t) ]} (u_x^+)^2(t,x) dx = \mu^+(\{\zeta\}),
\end{equation*}
which in general does not vanish.
\end{remark}
\begin{proof}[Proof of Lemma \ref{Lem_44}]
Define $$I:=\inf\{t: t>t_0 \mbox{ and there exists a characteristic } \zeta(\cdot) \in A_{\zeta} \mbox{ such that } \zeta(t) <\zeta^r(t)\},$$ where $A_\zeta$ is given by Definition \ref{Def_Azeta}. If $I > t_0$ then the proof is trivial, since  
$(\zeta^{rl}(t), \zeta^r(t) ] = \emptyset$ for all $t<I$. So, suppose $I=t_0$. Then there exists a sequence $\delta_n$ such that $\delta_n \to 0$ as $n \to \infty$ and for every $n>0$ there exists a characteristic $\zeta(\cdot) \in A_{\zeta}$ such that $$\delta_n= \inf\{t-t_0 : \zeta(t) < \zeta^r(t)\}.$$ In other words, timepoint $t_0 + \delta_n$ is a branching point for some characteristic $\zeta(\cdot) \in A_{\zeta}$. Let $\beta^{\delta_n}$ be the characteristic defined by
\begin{equation*}
\beta^{\delta_n}(\cdot) := \inf \{\zeta(\cdot): \zeta(t)=\zeta^r(t) \mbox{ for every } t \in [t_0,t_0 + \delta_n] \}.
\end{equation*}
We interpret $\beta^{\delta_n}$ as the leftmost characteristic from the point of branching. With such choice of $\beta^{\delta_n}$ and $\Delta:= \{\delta_n\}_{n=1}^{\infty}$ as well as $\beta(\cdot) \equiv \zeta^{rl}(\cdot)$ the assumptions of Lemma \ref{Lem18} are satisfied. Using Lemma \ref{Lem18} and the property 
$$\int_{[\zeta^{rl}(t),\zeta^r(t)]}(u_x^+)^2(t,x)dx = \int_{[\zeta^{rl}(t),\zeta^{\delta_n}(t)]}(u_x^+)^2(t,x)dx$$
for $t \in [t_0,t_0+\delta_n]$ we infer \eqref{Eq_contleftright}. The proof of \eqref{Eq_contrightleft} is analogous.
\end{proof}
\begin{lemma}
\label{Lem14}
Let $u$ be a weak solution of \eqref{eq_WeakCH1}-\eqref{eq_WeakCH3}. Let $B = (\alpha,\beta)$ be a bounded open interval. Then for $t$ sufficiently close to $t_0$ (so that $\alpha^r(\cdot)$ and $\beta^l(\cdot)$ do not cross)
\begin{equation*}
(\alpha^r(t),\beta^l(t))\subset B(t)\subset [\alpha^{rl}(t),\beta^{lr}(t)],
\end{equation*}
where $B(t)$ is the thick pushforward, $\alpha^r$ is the rightmost characteristic with $\alpha^r(t_0)=\alpha$, $\beta^l$ is the leftmost characteristic with $\beta^l(t_0)=\beta$, $\alpha^{rl}$ is the left-rightmost characteristic with $\alpha^{rl}(t_0)=\alpha$ and $\beta^{lr}$ is the right-leftmost characteristic satsifying $\beta^{lr}(t_0) = \beta$.
\end{lemma}
\begin{proof}
If a characteristic emanating from $B=(\alpha,\beta)$ crosses $\alpha^r$ or $\beta^l$, it does so at a timepoint strictly bigger then $t_0$. This implies the second inclusion. The first inclusion can be proven as follows.
Take $\zeta \in (\alpha^r(t),\beta^l(t))$ and consider the leftmost backward characteristic $\zeta^{lb}$ with $\zeta^{lb}(t)=\zeta$. Clearly, $\zeta^{lb}(\cdot)$ does not cross $\alpha^r(\cdot)$ on $[t_0,t]$ due to the rightmost property of $\alpha^r(\cdot)$. By the same token, $\zeta^{lb}(\cdot)$ does not cross $\beta^l(\cdot)$ on $[t_0,t]$. Hence, $\alpha<\zeta^{lb}(t_0)<\beta$ and thus $\zeta=\zeta^{lb}(t) \in B(t)$.
\end{proof}
\begin{proposition}
\label{Prop16}
Let $u$ be a weak solution of \eqref{eq_WeakCH1}-\eqref{eq_WeakCH3} and let $B = (\alpha,\beta)$ be an open interval. Fix $t_0\ge 0$. Then 

\begin{equation}
\label{eq_opene1}
\lim_{t \to t_0^+} \int_{(\alpha^r(t),\beta^l(t))} (u_x^+)^2 (t,x) dx = \lim_{t \to t_0^+} \int_{B(t)} (u_x^+)^2 (t,x) dx,
\end{equation}
where $B(t)$ is the thick pushforward.
Consequently,
\begin{equation}
\label{eq_opene2}
\mu^+(t_0,(\alpha,\beta)) = \lim_{t \to t_0^+} \int_{B(t)} (u_x^+)^2 (t,x) dx - \int_{(\alpha,\beta)} (u_x^+)^2(t_0,x)dx.
\end{equation}
\end{proposition}

\begin{proof}
We first prove that
\begin{equation*}
\lim_{t \to t_0^+} \int_{(\alpha^r(t),\beta^l(t))} {(u_x^+)}^2 (t,x) dx \ge \limsup_{t \to t_0^+} \int_{B(t)} {(u_x^+)}^2 (t,x) dx.
\end{equation*}
Indeed, $B(t)\subset [\alpha^{rl}(t),\beta^{lr}(t)]$ by Lemma \ref{Lem14}. Hence,
\begin{equation*}
\limsup_{t \to t_0^+} \int_{B(t)} {(u_x^+)}^2 (t,x) dx \le \limsup_{t \to t_0^+} \int_{[\alpha^{rl}(t),\beta^{lr}(t)]} {(u_x^+)}^2 (t,x) dx = \lim_{t \to t_0^+} \int_{(\alpha^r(t),\beta^l(t)} {(u_x^+)}^2 (t,x) dx,
\end{equation*}
where the last equality follows by Lemma \ref{Lem_44}. The reverse inequality 
\begin{equation*}
\lim_{t \to t_0^+} \int_{(\alpha^r(t),\beta^l(t)} (u_x^+)^2 (t,x) dx \le \liminf_{t \to t_0^+} \int_{B(t)} (u_x^+)^2 (t,x) dx
\end{equation*}
is clear since $(\alpha^r(t),\beta^l(t)) \subset B(t)$ by Lemma \ref{Lem14}. This proves \eqref{eq_opene1}. Formula \eqref{eq_opene2} follows by Corollary \ref{Cor_muplus}.
\end{proof}

\noindent{\bf ${\bf B}$ - arbitrary open set}

\begin{proposition}
Let $u$ be a weak solution of \eqref{eq_WeakCH1}-\eqref{eq_WeakCH3} and let $B = \bigcup_{n=1}^{\infty} (\alpha_n,\beta_n)$ be an arbitrary bounded open subset in $\mathbb{R}$.
Then 
\begin{equation}
\label{Eq_cont6}
\mu^+(t_0,B) = \lim_{t \to t_0^+} \int_{B(t)} (u_x^+)^2(t,x) dx - \int_B (u_x^+)^2(t_0,x) dx,
\end{equation} 
where $B(t)$ is the thick pushforward of the set $B$. Similarly,
\begin{equation}
\label{Eq_cont7}
\mu^+(t_0,B) = \lim_{t \to t_0^+} \int_{\bigcup_{n=1}^{\infty} (\alpha_n^r(t),\beta_n^l(t))  } (u_x^+)^2 (t,x)dx - \int_B (u_x^+)^2(t_0,x) dx.
\end{equation} 
\end{proposition}
\begin{proof}

Recall that by Lemma \ref{Lem_supinfChar} the suprema and infima of families of characteristics are again characteristics. Using this property we define characteristics $\alpha_n(t)$ and $\beta_n(t)$ recursively as follows. 
\begin{enumerate}
\item $\alpha_1(t):= \inf_{\alpha \in (\alpha_1,\beta_1)} \{\alpha^{l}(t): \alpha^l(t_0) = \alpha\}$.
\item $\beta_1(t):= \sup_{\beta \in (\alpha_1,\beta_1)} \{\beta^{r}(t): \beta^r(t_0) = \beta\}$.
\item To define $\alpha_n(\cdot)$ let 
$$\alpha^*_n(t):= \inf_{\alpha \in (\alpha_n,\beta_n)} \{\alpha^{l}(t): \alpha^l(t_0) = \alpha\}$$
Suppose there exist $n_*,n_{**}<n$ such that 
$\beta_{n_*} \le \alpha_n < \alpha_{n_{**}}$ and 
\begin{eqnarray*}
\beta_{n_*} = \max\{\beta_{n'}: n'<n, \beta_{n'} \le \alpha_n\},\\
\alpha_{n_{**}} = \min\{\alpha_{n'}: n'<n, \alpha_{n'} > \alpha_n\}.
\end{eqnarray*}
Let
\begin{eqnarray*}
t_* &:=& \inf\{t > t_0: \alpha_n^{*}(t) < \beta_{n_*}(t)\},\\
t_{**} &:=& \inf\{t > t_0: \alpha_n^{*}(t) = \alpha_{n_{**}}(t)\},
\end{eqnarray*}
where both $t_*$ and $t_{**}$ can assume the value $+\infty$.
Define (see Fig. \ref{Fig7})
$$\alpha_n(t):=
\begin{cases} 
\alpha_n^{*}(t) &\mbox{ for } t\le \min(t_*,t_{**}), \\
\begin{cases}
\beta_{n_*}(t) &\mbox{ if } t_* \le t_{**}\\
\alpha_{n_{**}}(t) &\mbox{ if } t_*>t_{**}
\end{cases}
&\mbox{ for } t> \min(t_*,t_{**}).
\end{cases} $$
The same defining formula holds if $n_*$ does not exist (then we put $t_*=\infty$ and $\beta_{n_*}$ is irrelevant) or $n_{**}$ does not exist (then we put $t_{**}=\infty$ and $\alpha_{n_{**}}$ is irrelevant).
\item To define $\beta_n(\cdot)$ let 
$$\beta^*_n(t):= \sup_{\beta \in (\alpha_n,\beta_n)} \{\beta^{r}(t): \beta^r(t_0) = \beta\}.$$
Suppose there exists $n_{**}<n$ such that 
$\beta_n \le \alpha_{n_{**}}$ and 
\begin{eqnarray*}
\alpha_{n_{**}} = \min\{\alpha_{n'}: n'<n, \alpha_{n'} \ge \beta_n\}.
\end{eqnarray*}
Let
\begin{eqnarray*}
t_{**} &:=& \inf\{t > t_0: \beta_n^{*}(t) > \alpha_{n_{**}}(t)\},
\end{eqnarray*}
where $\alpha_n(t)$ has been defined in the previous step and both $t_*$ and $t_{**}$ can assume the value $+\infty$.
Define (see Fig. \ref{Fig7})
$$\beta_n(t):=
\begin{cases} 
\beta_n^{*}(t) &\mbox{ for } t\le t_{**}, \\
\alpha_{n_{**}}(t)&\mbox{ for } t> t_{**}.
\end{cases} $$
The same defining formula holds if $n_{**}$ does not exist (then we put $t_{**}=\infty$ and $\alpha_{n_{**}}$ is irrelevant).

\begin{figure}[h!]
\center
\includegraphics[width=14cm]{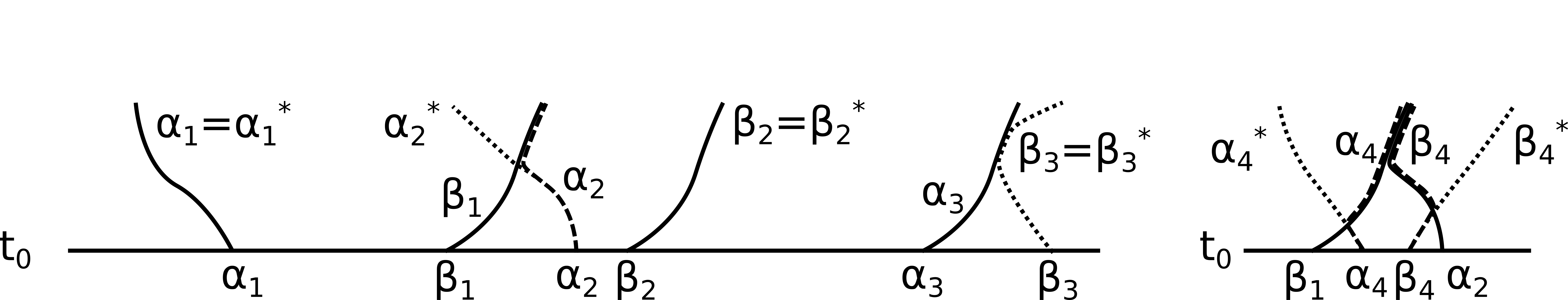}
\caption{Illustration of the process of definition of characteristics $\alpha_n$, $\beta_n$. Characteristics $\alpha_1$ and $\beta_1$ are defined without any intersection and thus coincide with $\alpha_1^*$ and $\beta_1^*$, respectively. Characteristic $\alpha_2^*$ (dotted line) crosses $\beta_1$ and thus $\alpha_2$ follows, from the timepoint of intersection on, characteristic $\beta_1$. Characteristic $\beta_n$ can also cross $\alpha_n$, which is the case for $\beta_3$. In this case, however $\beta_n$ does not follows $\alpha_n$, yet remains $\beta_n^*$. An example of a more complex situation when both $\alpha_n$ and $\beta_n$ cross other characteristics is shown in the right panel for $\alpha_4$ and $\beta_4$. Here $\alpha_4,\beta_4$ follow, beginning from some timepoint, the characterstics $\beta_1$ and $\alpha_2$, respectively. As $\beta_1$ and $\alpha_2$ coincide from some timepoint on, so do $\alpha_4$ and $\beta_4$. For clarity of the picture, the intersections of characteristics are depicted as transversal. In fact, however, the characteristics can only intersect tangentially.}
\label{Fig7}
\end{figure}
\end{enumerate}
The definition is designed in such a way that if $\alpha_{n_1}<\beta_{n_1} \le \alpha_{n_2}<\beta_{n_2}$ for some $1 \le n_1,n_2 \le N$ then 
\begin{equation*}
\alpha_{n_1}(t)<\beta_{n_1}(t) \le \alpha_{n_2}(t)<\beta_{n_2}(t)
\end{equation*} 
for every $t \ge t_0$ and
\begin{equation*}
\bigcup_{n=1}^{\infty} (\alpha_n(t),\beta_n(t)) \subset B(t) \subset \bigcup_{n=1}^{\infty} [\alpha_n(t),\beta_n(t)] 
\end{equation*}
(compare Lemma \ref{Lem14}). Hence, up to a set of measure $0$ we have
\begin{equation*}
B(t) = \bigcup_{n=1}^{\infty} (\alpha_n(t),\beta_n(t)).
\end{equation*}
and the union above is disjoint. Take now 
${\mathcal B}^{1/N}:= \bigcup_{n=N}^{\infty} \{\gamma(\cdot): \gamma(t_0) \in (\alpha_n,\beta_n) \mbox{ and } \gamma(t) \in [\alpha_n(t),\beta_n(t)] \mbox{ for every } t \in (t_0,\infty) \}.$
Observe that the family $\{{\mathcal B}^{\delta}\}_{\delta \in \Delta}$ with $\Delta= \{1/N: N=1,2,\dots\}$ fulfils the assumptions of Lemma \ref{Lem19}. Hence,
\begin{equation}
\label{Eq_contconv0}
\lim_{N \to \infty} \left( \limsup_{t \to t_0^+} \int_{B^{1/N}(t)} (u_x^+)^2(t,x) dx\right) = 0.
\end{equation}
Let now
\begin{equation*}
B_N:= \bigcup_{n=1}^{N} (\alpha_n,\beta_n).
\end{equation*}
For every $n \in \{1,2, \dots,N\}$ and every $t \ge t_0$ we have $$(\alpha_n^{r}(t),\beta_n^l(t)) \subset (\alpha_n(t),\beta_n(t)) \subset (\alpha_n,\beta_n)(t) \subset [\alpha_n^{rl}(t),\beta_n^{lr}(t)],$$
where the first inclusion follows from the definition of rightmost and leftmost characteristics, the second inclusion follows by the definition of $\alpha_n(t),\beta_n(t)$ and the last inclusion is due to Lemma \ref{Lem14}. Consequently, by Proposition \ref{Prop16} and Lemma \ref{Lem_44},
\begin{eqnarray*}
\mu^+\left(t_0,(\alpha_n,\beta_n)\right) = \lim_{t \to t_0^+} \int_{(\alpha_n(t),\beta_n(t))} (u_x^+)^2 (t,x) dx - \int_{(\alpha_n,\beta_n)} (u_x^+)^2(t_0,x)^2dx.
\end{eqnarray*}
Hence, for every $N<\infty$
\begin{eqnarray*}
\mu^+\left(t_0,B_N\right) = \lim_{t \to t_0^+} \int_{\bigcup_{n=1}^{N} (\alpha_n(t),\beta_n(t))} (u_x^+)^2 (t,x) dx - \int_{B_N} (u_x^+)^2(t_0,x)dx.
\end{eqnarray*}
Passing to the limit $N \to \infty$ we obtain:
\begin{equation*}
 \lim_{N \to \infty}\lim_{t \to t_0^+} \int_{\bigcup_{n=1}^{N} (\alpha_n(t),\beta_n(t))} (u_x^+)^2 (t,x) dx = \mu^+\left(t_0,B\right) + \int_{B} (u_x^+)^2(t_0,x)dx.
\end{equation*}
Moreover, by \eqref{Eq_contconv0} we have
\begin{equation*}
\lim_{N \to \infty} \left( \limsup_{t \to t_0^+} \int_{\bigcup_{n=N+1}^{\infty} (\alpha_n(t),\beta_n(t))} (u_x^+)^2(t,x) dx\right) = 0.
\end{equation*}
Fix $\epsilon>0$. By the two above convergences there exists $N_{\epsilon}$ so large that there exists $t_\epsilon>t_0$ such that for $t \in [t_0,t_\epsilon]$
\begin{eqnarray*}
\left|\int_{\bigcup_{n=1}^{N_\epsilon} (\alpha_n(t),\beta_n(t))} (u_x^+)^2 (t,x) dx - \left[\mu^+\left(t_0,B\right) + \int_{B} (u_x^+)^2(t_0,x)dx\right]\right| < \epsilon/2,\\
\int_{\bigcup_{n=N^{\epsilon}+1}^{\infty} (\alpha_n(t),\beta_n(t))} (u_x^+)^2(t,x) dx < \epsilon/2.
\end{eqnarray*}
Hence, 
\begin{equation*}
\left|\int_{\bigcup_{n=1}^{\infty} (\alpha_n(t),\beta_n(t))} (u_x^+)^2 (t,x) dx - \left[\mu^+\left(t_0,B\right) + \int_{B} (u_x^+)^2(t_0,x)dx\right]\right| < \epsilon,
\end{equation*}
which means that for every $\epsilon>0$ there exists $t_\epsilon>t_0$ such that for $t \in [t_0,t_\epsilon]$
\begin{equation*}
\left|\int_{B(t)} (u_x^+)^2 (t,x) dx - \left[\mu^+\left(t_0,B\right) + \int_{B} (u_x^+)^2(t_0,x)dx\right]\right| < \epsilon.
\end{equation*}
This concludes the proof of \eqref{Eq_cont6}.
The proof of \eqref{Eq_cont7} is similar and in fact simpler since the sets $(\alpha_n^r(t),\beta_n^l(t))$ are pairwise disjoint for different $n$.
\end{proof}

\begin{remark}
Characteristic $\alpha_n^*$ can be interpreted as follows. Let $$\tau:= \inf\{t>t_0: \mbox{ there exists } \alpha \in (\alpha_n,\beta_n) \mbox{ such that } \alpha^l(t)=\alpha_n^r(t)\}.$$
There are now two possibilities. Either there exists $\alpha \in (\alpha_n,\beta_n)$ such that $\alpha(\tau)=\alpha_n^r(\tau)$ and then
\begin{equation*}
\alpha_n^*(t) = 
\begin{cases}
\alpha_n^r(t) &\mbox{ for } t_0 \le t \le \tau,\\
(\alpha_n^r(\tau))^{l}(t) &\mbox{ for } t > \tau
\end{cases}
\end{equation*}
or, otherwise,
\begin{equation*}
\alpha_n^*(t) = 
\begin{cases}
\alpha_n^r(t) &\mbox{ for } t_0 \le t \le \tau,\\
(\alpha_n^r(\tau))^{rl}(t) &\mbox{ for } t > \tau.
\end{cases}
\end{equation*}
\end{remark}

\noindent{\bf ${\bf B}$ - arbitrary compact set}
\begin{lemma}
\label{Lem118}
Let $u$ be a weak solution of \eqref{eq_WeakCH1}-\eqref{eq_WeakCH3} and let $K \subset \mathbb{R}$ be an arbitrary compact set. Fix $t_0,T$ such that $0 \le t_0 < T$. Let $U=(\underline{U},\overline{U})$ be a bounded open interval satisfying $K \subset U$ and so large that 
\begin{equation}
\label{Eq_supKinfK}
\begin{cases}
[\sup(K)+2(T-t_0)\sup(|u|)] \in U,\\ [\inf(K)-2(T-t_0)\sup(|u|)] \in U.
\end{cases}
\end{equation}
Let the open set $U \backslash K$ have the representation
\begin{equation*}
U \backslash K = \bigcup_{n=1}^{\infty}   (\alpha_n,\beta_n),
\end{equation*}
where the open intervals $(\alpha_n,\beta_n)$ are pairwise disjoint. Then for $t\in[t_0,T]$ 
\begin{equation}
\label{eq_Koft}
K(t) = (\underline{U}^r(t),\overline{U}^l(t)) \backslash \bigcup_{n=1}^{\infty}   (\alpha^r_n(t),\beta^l_n(t)),
\end{equation}
where we adopt the convention $(\alpha^r_n(t),\beta^l_n(t)) = \emptyset$ if $\alpha^r_n(t)>\beta^l_n(t)$.
\end{lemma}
\begin{proof}
First we prove the inclusion $\subset$ in \eqref{eq_Koft}. Let $x \in K(t)$. Condition \eqref{Eq_supKinfK} guarantees that $$K(t)\subset (\underline{U}^r(t),\overline{U}^l(t))$$ for $t \in [t_0,T]$ and hence $x \in (\underline{U}^r(t),\overline{U}^l(t))$. Fix now $n$. We will show that $x \notin (\alpha_n^r(t),\beta_n^l(t)$. Indeed, since $x\in K(t)$ there exists $\zeta \in K$ such that $\zeta(t) = x$. As $(\alpha_n,\beta_n)\cap K = \emptyset$ then either $\beta_n \le \zeta$ or $\alpha_n\ge \zeta$.  
In the former case we have $\beta_n^l(t)\le \zeta^l(t)\le \zeta(t)=x$ since $\beta_n^l$ is the leftmost characteristic. In the latter case, similarly, we have $\alpha_n^r(t)\ge \zeta^r(t)\ge \zeta(t)=x$.

To prove the reverse inclusion, let $x \in (\underline{U}^r(t),\overline{U}^l(t)) \backslash \bigcup_{n=1}^{\infty}   (\alpha^r_n(t),\beta^l_n(t))$. Then there exists $\zeta \in U$ and a characteristic $\zeta(\cdot)$ such that $\zeta(t)=x$. If  $\zeta \in K$ then $x \in K(t)$ and the proof is finished. If $\zeta \notin K$ then $\zeta \in (\alpha_n,\beta_n)$ for some $n$. Since, however, $\zeta(t) \notin (\alpha_n^r(t),\beta_n^l(t))$, there exists time $\tau \in (t_0,t]$ such that either $\zeta(\tau) = \alpha_n^r(\tau)$ (see Fig. \ref{Fig8}) or $\zeta(\tau) = \beta_n^l(\tau)$. In the former case we observe that due to \eqref{Eq_supKinfK}, $\alpha_n \neq  \underline{U}$ and hence $\alpha_n \in K$. 
The characteristic defined by
\begin{equation*}
\alpha_n(s) = \begin{cases}
\alpha_n^r(s) &\mbox{ if } s \in [t_0,\tau],\\
\zeta(t) &\mbox{ if } s\in (\tau,t]
\end{cases}
\end{equation*}
satisfies $\alpha_n(t_0) \in K$ and $\alpha_n(t)=x$. Hence, $x \in K(t)$. The proof in the case $\zeta(\tau) = \beta_n^l(\tau)$ is analogous.

\begin{figure}[h!]
\center
\includegraphics[width=6cm]{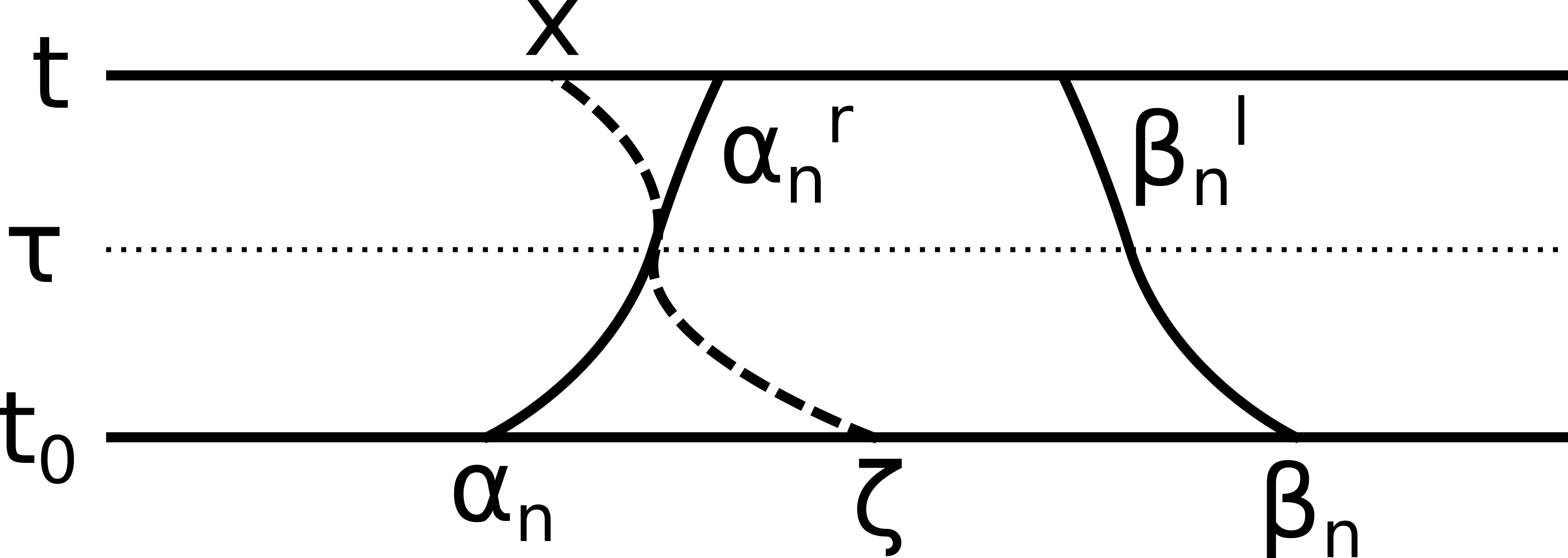}
\caption{Illustration of the proof of Lemma \ref{Lem118}. If $\zeta \in (\alpha_n,\beta_n)$ and $x = \zeta(t) \notin (\alpha_n^r(t),\beta_n^l(t))$ then, due to geometrical reasons, $\alpha_n^r(\cdot)$ and $\zeta(\cdot)$ have to cross. Hence, the concatenation of $\alpha_n^r(\cdot)$ (until the crossing time $\tau$) and $\zeta(\cdot)$ (from $\tau$ on) is an example of characteristic emanating from $K$ and ending in $x$.}
\label{Fig8}
\end{figure}

\end{proof}

\begin{proposition}
Let $u$ be a weak solution of \eqref{eq_WeakCH1}-\eqref{eq_WeakCH3} and let $K$ be an arbitrary compact subset of $\mathbb{R}$. Fix $t_0 \ge 0$.
Then 
\begin{equation}
\label{Eq_cont8}
\mu^+(t_0,K) = \lim_{t \to t_0^+} \int_{K(t)} (u_x^+)^2 dx - \int_K (u_x^+)^2(t_0,x) dx,
\end{equation} 
where $K(t)$ is the thick pushforward.
\end{proposition}

\begin{proof}
Let $U = (\underline{U},\overline{U})$ be a bounded open interval such that $K \subset U$ and so large that the assumptions of Lemma \ref{Lem118} are satisfied. By Corollary \ref{Cor_muplus}

\begin{equation*}
\mu^+(t_0,U) = \lim_{t \to t_0^+} \int_{(\underline{U}^r(t),\overline{U}^l(t))} (u_x^+)^2 (t,x) dx - \int_{U} (u_x^+)^2(t_0,x)dx.
\end{equation*}
On the other hand, $U = K \cup A$ where $A = U\backslash K =  \bigcup_{n=1}^{\infty} (\alpha_n,\beta_n)$. Hence, by \eqref{Eq_cont7}

\begin{equation*}
\mu^+(t_0,A) = \lim_{t \to t_0^+} \int_{\bigcup_{n=1}^{\infty} (\alpha_n^r(t),\beta_n^l(t))  } (u_x^+)^2(t,x) dx - \int_A (u_x^+)^2(t_0,x) dx.
\end{equation*} 
Subtracting these two equalities and using Lemma \ref{Lem118} we conclude.
\end{proof}

\noindent{\bf ${\bf B}$ - arbitrary Borel set}
\begin{proof}[Proof of Theorem \ref{thB}]
Fix $\epsilon>0$ By regularity of the measure $\mu^+$ there exists a compact set $K^{\epsilon}$ and an open set $A^{\epsilon}$ such that 
\begin{equation*}
K^\epsilon \subset B \subset A^\epsilon
\end{equation*}
and 
\begin{eqnarray*}
\int_{A^\epsilon \backslash B} u_x^+(t_0,x)^2 dx < \epsilon,\\
\mu^+(t_0,A^\epsilon \backslash B) < \epsilon,\\
\int_{B \backslash K^\epsilon} u_x^+(t_0,x)^2 dx < \epsilon,\\
\mu^+(t_0,B \backslash K^\epsilon) < \epsilon.
\end{eqnarray*}
We calculate
\begin{eqnarray*}
\limsup_{t \to t_0^+} \int_{B(t)} (u_x^+)^2(t,x)dx &\le& \lim_{t \to t_0^+} \int_{A^{\epsilon}(t)} (u_x^+)^2(t,x)dx\\ &=& \mu^+(t_0,A^{\epsilon}) + \int_{A^{\epsilon}} (u_x^+)^2(t_0,x) dx\\ &\le& \mu^+(t_0,B) + \epsilon + \int_{B} (u_x^+)^2(t_0,x) dx + \epsilon,
\end{eqnarray*}
where we used \eqref{Eq_cont6} to pass to the limit with $t$.
Similarly,
\begin{eqnarray*}
\liminf_{t \to t_0^+} \int_{B(t)} (u_x^+)^2(t,x)dx &\ge& \lim_{t \to t_0^+} \int_{K^{\epsilon}(t)} (u_x^+)^2(t,x)dx\\ &=& \mu^+(t_0,K^{\epsilon}) + \int_{K^{\epsilon}} (u_x^+)^2(t_0,x) dx\\ &\ge& \mu^+(t_0,B) - \epsilon + \int_{B} (u_x^+)^2(t_0,x) dx - \epsilon,
\end{eqnarray*}
where we used \eqref{Eq_cont8} to pass to the limit with $t$.  Passing $\epsilon \to 0$ we conclude.
\end{proof}

\section{Proofs of Theorems \ref{Th_muac0}, \ref{th_mu0dis}, \ref{Th_countably} and \ref{Th_nomaxdissip}}
\label{Sec_proofsConclusions}
\begin{proof}[Proof of Theorem \ref{Th_muac0}]
Decompose $$\mu^+(t_0,dx) = g \mathcal{L}^1(dx) + (\mu^+)^{sing}(t_0,dx)$$ where $g$ is the density of $(\mu^+(t_0,dx))^{ac}$ with respect to the one-dimensional Lebesgue measure and $(\mu^+)^{sing}$ denotes the singular, with respect to $\mathcal{L}^1$, part of measure $\mu^+(t_0,dx)$. Suppose $g \neq 0$. Then there exists a Borel set $B$ and $\epsilon > 0$ such that 
\begin{eqnarray*}
(\mu^+)^{sing}(t_0,B)  &=& 0,\\
\mathcal{L}^1(B) &\ge& \epsilon, \\
g &\ge& \epsilon \mbox{ on } B.
\end{eqnarray*}
Furthermore, due to Proposition \ref{Prop_MainCH} there exists a subset $B_1 \subset B$ and $\tau>0$ such that 
\begin{itemize}
\item $\mathcal{L}^1(B_1) > \epsilon / 2$,
\item the characteristics originating in $B_1$ are unique on $[t_0,t_0+\tau]$ 
\end{itemize}
and for every $\zeta \in B_1$  and  $t \in [t_0,t_0+\tau]$
\begin{itemize}
\item $|u_x(t,\zeta(t))| < 1/\tau$, where $\zeta(\cdot)$ is the unique characteristic satisfying $\zeta(t_0)=\zeta$,
\item the differential equation
\begin{equation*}
\dot{v} = u^2 - \frac 1 2 v^2 - P,
\end{equation*} 
is satisfied, where $v=v(t)=u_x(t,\zeta(t))$, $u=u(t)=u(t,\zeta(t))$, $P=P(t)=P(t,\zeta(t))$.
\end{itemize} 
The last condition  implies, in particular, that for every $\zeta \in B_1$ we have 
\begin{equation*}
\lim_{t \to t_0^+}  u_x^+(t,\zeta(t)) = u_x^+(t_0,\zeta(t_0)).
\end{equation*}
Consequently, using the representation formula from Theorem \ref{thB} and the change of variables formula from Proposition \ref{Prop_ChangeOfV}
we obtain
\begin{eqnarray*}
\mu^+(t_0,B_1) &=& \lim_{t \to t_0^+} \int_{B_1(t)} (u_x^+)^2(t,x)dx - \int_{B_1}(u_x^+)^2(t_0,x)dx \\
&=& \lim_{t \to t_0^+} \int_{B_1} (u_x^+)^2(t,\zeta(t))  e^{\int_{t_0}^{t} u_x(s,\zeta(s))ds}d\zeta - \int_{B_1}(u_x^+)^2(t_0,\zeta) d\zeta \\
  &=& 0, 
\end{eqnarray*}
where the last equality follows by the Lebesgue dominated convergence theorem and the bound
$$(u_x^+)^2(t,\zeta(t)) e^{\int_{t_0}^{t} u_x (s,\zeta(s))ds} \le \tau^{-2} e^{\tau/\tau} = \tau^{-2} e,$$ which holds for every $\zeta \in B_1$ and $t \in [t_0,t_0+\tau]$.
On the other hand,
$$\mu^+(t_0,B_1) \ge \int_{B_1} g(x) dx \ge \epsilon^2/2,$$ which gives contradiction and concludes the proof for $\mu^+$. The proof for $\mu^-$ follows dually by considering the 'backward' solution $u^{t_0 b}(t,x):=-u(t_0-t,x)$ and using the result for $\mu^+$.
\end{proof}

\begin{proof}[Proof of Theorem \ref{th_mu0dis}]
We distinguish two cases, $t_0=0$ and $t_0>0$ as the proofs in those cases are different. Both proceed by obtaining a contradiction if we assume that $\mu^+(t_0,\cdot) \neq 0$, however in the former we obtain a contradiction with the weak energy condition and in the latter with the Oleinik-type criterion from Definition \ref{Def_dissipative}.

{\bf Case} \noindent $\bf{t_0=0}$. 
Suppose $\mu^+(t_0,\mathbb{R}) = S>0$. For a test function $\phi \in C_c(\mathbb{R})$ we calculate, using Theorem \ref{Th_Cadlag} and Definition \ref{Def_MAD},
\begin{eqnarray*}
\lim_{t \to t_0^+} \int_{\mathbb{R}} \phi(x)(u_x)^2(t,x)dx &=& \lim_{t \to t_0^+} \int_{\mathbb{R}} \phi(x)(u_x^+)^2(t,x)dx+\lim_{t \to t_0^+} \int_{\mathbb{R}} \phi(x)(u_x^-)^2(t,x)dx \\
&=& \int_{\mathbb{R}} \phi(x)(u_x)^2(t_0,x)dx+\int_{\mathbb{R}} \phi(x) \mu^+(t_0,dx). 
\end{eqnarray*}
Take $\phi = \phi^K$ of the following form:
\begin{equation*}
\phi^K(x) := \begin{cases}
0 &\mbox{ if } x \in (-\infty,-K-1),\\
x+K+1 &\mbox{ if } x \in [-K-1,-K),\\
1 &\mbox{ if } x \in [-K,K),\\
K+1-x &\mbox{ if } x \in [K,K+1),\\
0 &\mbox{ if } x \in [K+1,\infty),
\end{cases}
\end{equation*}
where $K$ is so big that 
\begin{eqnarray*}
\int_{\mathbb{R}} \phi(x)\mu^+(t_0,dx) \ge S-S/4, \\ 
\int_{\mathbb{R}} \phi(x)(u_x)^2(t_0,x)dx \ge \int_{\mathbb{R}} (u_x)^2(t_0,x)dx - S/4, \\ 
\int_{\mathbb{R}} \phi(x)u^2(t_0,x)dx \ge \int_{\mathbb{R}} u^2(t_0,x)dx - S/4.
\end{eqnarray*}
Then
\begin{eqnarray*}
\liminf_{t \to t_0^+} \int_{\mathbb{R}}(u_x)^2(t,x)dx \ge \lim_{t \to t_0^+} \int_{\mathbb{R}} \phi(x)(u_x)^2(t,x)dx &\ge& \int_{\mathbb{R}} (u_x)^2(t_0,x)dx- S/4 + (S - S/4)
\end{eqnarray*}
and 
\begin{eqnarray*}
\liminf_{t \to t_0^+} \int_{\mathbb{R}}u^2(t,x)dx \ge \lim_{t \to t_0^+} \int_{\mathbb{R}} \phi(x)u^2(t,x)dx = \int_{\mathbb{R}} \phi(x)u^2(t_0,x)dx \ge \int_{\mathbb{R}} u^2(t_0,x)dx - S/4.
\end{eqnarray*}
Combination of these two convergences leads to
\begin{equation*}
\liminf_{t \to t_0^+} \int_{\mathbb{R}} [(u_x)^2(t,x) + u^2(t,x)] dx \ge \int_{\mathbb{R}} [(u_x)^2(t_0,x) + u^2(t_0,x)] dx + S/4,
\end{equation*}
which contradicts the weak energy condition from Definition \ref{Def_dissipative} and hence implies that $u$ is not dissipative.

{\bf Case} ${\bf t_0>0}$. 
Since $\mu^+(t_0,x) \perp \mathcal{L}^1$ by Theorem \ref{Th_muac0}, there exists a set $C$ (without loss of generality bounded) satisfying 
\begin{itemize}
\item $\mathcal{L}^1(C)=0,$
\item $\mu^+(t_0,\mathbb{R} \backslash C)=0.$
\end{itemize}
Suppose $\mu^+(t_0,C)=S>0$ and fix $\epsilon>0$. By regularity of the measure $\mathcal{L}^1$ there exists an open set $U = \bigcup_{n=1}^{\infty} (\alpha_n,\beta_n)$ such that
\begin{itemize}
\item $C \subset U,$
\item $\mathcal{L}^1(U) < \epsilon,$
\item $(\alpha_n,\beta_n)$ are pairwise disjoint.
\end{itemize}
Let 
$N$ be so large that $$\mu^+\left(t_0,\bigcup_{n=1}^{N} (\alpha_n,\beta_n)\right) > S/2.$$
Next, take $\delta$ so small that for $U^{\delta}_N := \bigcup_{n=1}^{N} (\alpha_n+\delta,\beta_n-\delta)$ we have $$\mu^+(t_0,U_N^{\delta})>S/4.$$
Taking
\begin{equation*}
\phi^\epsilon(x) := \begin{cases}
1 &\mbox{ if } x \in U_N^{\delta},\\
(x-\alpha_n)/\delta &\mbox{ if } x \in (\alpha_n, \alpha_n+\delta],\\
(\beta_n-x)/\delta &\mbox{ if } x \in [\beta_n-\delta, \beta_n),\\
0 &\mbox{ otherwise}
\end{cases}
\end{equation*}
we obtain
\begin{eqnarray}
\label{Eq_limS2}
\lim_{t \to t_0^+} \int_{\mathbb{R}} (u_x^+)^2(t_0,x) \phi^{\epsilon}(x) dx &=& \int_{\mathbb{R}} \phi^{\epsilon}(x)d\mu^+(t_0,x) + \int_{\mathbb{R}} (u_x^+)^2 (t_0,x) \phi^{\epsilon}(x) dx \nonumber\\
 &\ge& S/4 + \int_{\mathbb{R}} (u_x^+)^2 (t_0,x) \phi^{\epsilon}(x) dx \ge S/4.
\end{eqnarray}
Noting that $0 \le \phi^{\epsilon}(x) \le 1$ for every $x \in \mathbb{R}$ and $\mathcal{L}^1 (\{x \in \mathbb{R}: \phi^{\epsilon}(x) \neq 0 \})< \epsilon$ we pass to the limit $\epsilon \to 0$ in \eqref{Eq_limS2} and obtain that $u_x^+$ is unbounded in the neighborhood of $\{t_0\} \times C$, which contradicts dissipativity.
\end{proof}

\begin{proof}[Proof of Theorem \ref{Th_countably}]
The function
$$t \mapsto \int_{[\alpha^l(t),\beta^r(t)]} (u_x^+)^2(t,x)dx$$
belongs to $BV_{loc}$ by \cite[Theorem 2.7]{CHGJ}. Hence, it has only at most countably many discontinuities and thus $\mu^+([\alpha^l(t),\beta^r(t)]) \neq 0$ for at most countably many $t$. Considering $\alpha_n := -n$ and $\beta_n :=n$ we obtain, using the $\sup(|u|)$ bound on the speed of characteristics, that $\mu^+([-n+\sup(u)t,n-\sup(u)t]) \neq 0$ for at most countably many $t$. Since this holds for every $n = 1,2, \dots$, we obtain $\mu^+(t,\mathbb{R}) \neq 0$ for at most countably many $t$.
\end{proof}

\begin{proof}[Proof of Theorem \ref{Th_nomaxdissip}]
By the theory of Bressan-Constantin \cite{BC} there exists a 'conservative' solution $\hat{u}$ of \eqref{eq_WeakCH1}-\eqref{eq_WeakCH3}, defined on $[t_0,\infty)\times \mathbb{R}$, which satisfies:
\begin{itemize}
\item $\hat{u}(t_0,\cdot) = u(t_0,\cdot)$,
\item $E(\hat{u}(t,\cdot)) \le E(\hat{u}(t_0,\cdot))$ for every $t \ge t_0$,
\item $E(\hat{u}(t,\cdot)) = E(\hat{u}(t_0,\cdot))$ for almost every $t \ge t_0$.
\end{itemize}
Let $\bar{u}$ be defined by 
\begin{equation*}
\bar{u}(t,\cdot) = \begin{cases}
u(t,\cdot) &\mbox{ for } t \in [0,t_0], \\
\hat{u}(t,\cdot) &\mbox{ for } t \in (t_0,\infty).
\end{cases}
\end{equation*}
and let $D>0$ be so large that for $B:=[-D,D]$ we have
\begin{equation*}
\int_B \frac 1 2 (u^2(t_0,x)+u_x^2(t_0,x))dx \ge E(u(t_0,\cdot)) - \frac {\mu^+(t_0,B)}{4}.
\end{equation*}
Then, observing that 
\begin{eqnarray*}
\lim_{t \to t_0^+} \int_{B(t)} u^2(t,x)dx &=&  \int_{B} u^2(t_0,x)dx,\\
\lim_{t \to t_0^+} \int_{B(t)} (u_x^-)^2(t,x)dx &=&  \int_{B} (u_x^-)^2(t_0,x)dx,\\
\lim_{t \to t_0^+} \int_{B(t)} (u_x^+)^2(t,x)dx &=&  \int_{B} (u_x^+)^2(t_0,x)dx + \mu^+(t_0,B),
\end{eqnarray*}
where the first convergence follows by continuity of $u$, the second by \cite[Theorem 2.6]{CHGJ} and the third by Theorem \ref{thB}, we obtain
\begin{eqnarray*}
\limsup_{t \to t_0^+}E(\bar{u}(t,\cdot)) &=&  \limsup_{t \to t_0^+}E(\hat{u}(t,\cdot)) \le E(u(t_0,\cdot))\\ 
&\le& \int_B \frac 1 2 (u^2(t_0,x)+u_x^2(t_0,x))dx + \frac {\mu^+(t_0,B)}{4} \\
&=& \lim_{t \to t_0^+} \int_{B(t)} \frac 1 2 (u^2(t,x) + u_x^2(t,x))dx - \frac{\mu^+(t_0,B)}{2} + \frac {\mu^+(t_0,B)}{4} \\
&\le& \liminf_{t \to t_0^+} E(u(t,\cdot)) - \frac {\mu^+(t_0,B)}{4}.
\end{eqnarray*}
\end{proof}

\end{document}